\documentclass[12pt]{amsart}
\usepackage[centertags]{amsmath}
\usepackage{amssymb,amscd,amsfonts,latexsym,a4wide,amsthm}
\advance\headheight by 3pt

\renewcommand{\subjclass}[1]{\thanks{\emph{2000 Mathematics Subject Classification:}~#1}}
\renewcommand{\keywords}[1]{\thanks{\emph{Keywords and Phrases:}~#1}}

\theoremstyle{plain}
\newtheorem{theorem}{Theorem}[section]
\newtheorem{lemma}{Lemma}[section]
\newtheorem{corollary}[theorem]{Corollary}
\newtheorem{proposition}[lemma]{Proposition}

\def\teto#1{\setbox\z@\hbox{${#1\vphantom k}$}\hbox{%
 \hbox{\lower2\ex@\hbox{\lower\dp\z@\hbox{\vbox{\hrule
 \hbox{\vrule\hskip2\ex@\vbox{\vskip2\ex@\box\z@\vskip1\ex@}%
 \hskip2\ex@\vrule}}}}}}}
\catcode`\@=\active

\newcommand{\Q}{\mathbb{Q}}

\newcommand{\Z}{\mathbb{Z}}

\newcommand{\y}{\textbf{y}}

\newcommand{\OO}{\mathcal{O}}

\newcommand{\PPP}{\mathcal{P}}

\newcommand{\ve}{\varepsilon}
\newcommand{\al}{\alpha}
\newcommand{\be}{\beta}

\newcommand{\fp}{\mathfrak{p}}

 \DeclareMathOperator{\ord}{ord}
\DeclareMathOperator{\id}{id}

\newcommand{\kdots}{,\ldots ,}
\renewcommand{\eqref}[1]{(\ref{#1})}

\newcommand{\Qq}{\mathbb{Q}}

\newcommand{\Zz}{\mathbb{Z}}

\newcommand{\GL}{{\rm GL}}
\newcommand{\half}{\textstyle{\frac{1}{2}}}
\newcommand{\quarter}{\textstyle{\frac{1}{4}}}

\title[Multiply monogenic orders]
{Multiply monogenic orders}% over finitely generated domains}

\author[A. B\'erczes]{Attila B\'erczes}
\thanks{The research was supported in part by grants T67580
and T75566 (A.B., K.G.) of the Hungarian National Foundation for
Scientific Research, and the J\'anos Bolyai Research Scholarship
(A.B.). The work is supported by the T\'AMOP
4.2.1./B-09/1/KONV-2010-0007 project. The project is implemented
through the New Hungary Development Plan, co-financed by the
European Social Fund and the European Regional Development Fund.
(A.B.)}
\address{A. B\'erczes \newline
         \indent Institute of Mathematics, University of Debrecen
     \newline
         \indent Number Theory Research Group,
     Hungarian Academy of Sciences and
     \newline
         \indent University of Debrecen
     \newline
         \indent H-4010 Debrecen, P.O. Box 12, Hungary}
\email{berczesa\char'100math.klte.hu}

\author[J.-H. Evertse]{Jan-Hendrik Evertse}
\address{J.-H. Evertse \newline
         \indent Universiteit Leiden, Mathematisch Instituut, \newline
         \indent Postbus 9512, 2300 RA Leiden, The Netherlands}
\email{evertse\char'100math.leidenuniv.nl}

\author[K. Gy\H{o}ry]{K\'{a}lm\'{a}n Gy\H{o}ry}
\address{K. Gy\H{o}ry \newline
         \indent Institute of Mathematics, University of Debrecen \newline
         \indent Number Theory Research Group, Hungarian Academy of Sciences and \newline
         \indent University of Debrecen \newline
         \indent H-4010 Debrecen, P.O. Box 12, Hungary}
\email{gyory\char'100math.klte.hu}

\begin{document}

\maketitle

\begin{abstract}
Let $A=\Zz [x_1\kdots x_r]\supset\Zz$ be a domain which is finitely generated
over $\Zz$ and integrally closed in its quotient field $L$.
Further, let $K$ be a finite extension field of $L$.
An $A$-order in $K$ is a domain $\OO\supset A$ with quotient field
$K$ which is integral over $A$. $A$-orders in $K$  of the type $A [\alpha ]$
are called monogenic.
It was proved by Gy\H{o}ry \cite{Gy17}
that for any given $A$-order $\OO$ in $K$
there are at most finitely many $A$-equivalence classes of $\alpha\in\OO$ with
$A[\alpha ]=\OO$,
where two elements $\alpha ,\beta$ of $\OO$ are
called $A$-equivalent if $\beta =u\alpha +a$ for some $u\in A^*$, $a\in A$.
If the number of $A$-equivalence classes of $\alpha$ with $A[\alpha ]=\OO$
is at least $k$, we call $\OO$ $k$ times monogenic.

In this paper we study orders which are more than one time
monogenic. Our first main result is that if $K$ is any finite
extension of $L$ of degree $\geq 3$, then there are only finitely
many three times monogenic $A$-orders in $K$. Next, we define two
special types of two times monogenic $A$-orders, and show that
there are extensions $K$ which have infinitely many orders of
these types. Then under certain conditions imposed on the Galois
group of the normal closure of $K$ over $L$, we prove that $K$ has
only finitely many two times monogenic $A$-orders which are not of
these types. Some immediate applications to canonical number
systems are also mentioned.
\end{abstract}

\section{Introduction}\label{1}

\setcounter{equation}{0}

In this introduction we present our results in the special case $A=\Zz$.
Our general results
over arbitrary finitely generated domains $A$ are stated in the next section.

Let $K$ be an algebraic number field of degree $d \geq 2$ with
ring of integers $\OO_K$. The number field $K$ is called {\it monogenic}
if $\OO_K = \Z[\al]$ for some $\al \in \OO_K$. This is equivalent
to the fact that $\{1,\al, \dots, \al^{d-1}\}$ forms a $\Zz$-module basis for $\OO_K$.
The existence of such a basis, called {\it power integral basis},
considerably facilitates the calculations in $\OO_K$ and the study of arithmetical
properties of $\OO_K$.

The quadratic and cyclotomic number fields are monogenic, but this is not
the case in general. Dedekind \cite{Dedekind1} gave the first example for
a non-monogenic number field.

More generally, an order $\OO$ in $K$,
that is a subring of $\OO_K$ with quotient field equal to $K$,
is said to be monogenic if $\OO=\Z[\al]$
for some $\al \in \OO$. Then for $\be=\pm \al +a$ with $a \in \Z$ we also have
$\OO=\Z[\be]$. Such elements $\al,\be$ of $\OO$ are called $\Z$-{\it equivalent}.

In this paper, we deal with the ``Diophantine equation"
\begin{equation}\label{1.1}
\Z [\al ]= \OO\ \ \mbox{in } \al\in\OO
\end{equation}
where $\OO$ is a given order in $K$. As has been explained above, the solutions of
\eqref{1.1} can be divided into $\Z$-equivalence classes.
It was proved by Gy\H ory \cite{Gy16}, \cite{Gy1}, \cite{Gy19}
that there are only finitely many
$\Z$-equivalence classes of $\al
\in \OO$ with \eqref{1.1}, and that a full system of
representatives for these classes can be determined effectively.
Evertse and Gy\H ory \cite{EGy1}
gave a uniform and explicit upper bound, depending only on $d=[K:\Q ]$,
for the number of $\Z$-equivalence
classes of such $\al$.
For various generalizations and effective versions, we refer to
Gy\H ory \cite{Gy18}.

In what follows, the following definition will be useful.
\\[0.2cm]
{\bf Definition.} An order $\OO$ is called \emph{$k$ times monogenic,}
if there are at least $k$ distinct $\Zz$-equivalence classes of $\alpha$ with \eqref{1.1},
in other words, if there are at least $k$ pairwise $\Zz$-inequivalent elements
$\alpha_1\kdots\alpha_k\in\OO$ such that
\[
\OO =\Zz [\alpha_1]=\cdots =\Zz [\alpha_k].
\]
Similarly, the order $\OO$ is called
\emph{precisely/at most $k$ times monogenic,}
if there are precisely/at most $k$ $\Zz$-equivalence classes of $\alpha$
with \eqref{1.1}.
\\[0.2cm]
It is not difficult to show that any order $\OO$ in a quadratic number field
is precisely one time monogenic, i.e., there exist $\alpha\in\OO$ with
\eqref{1.1}, and these $\alpha$
are all $\Zz$-equivalent to one another.

Our first result is as follows.

\begin{theorem}\label{T1.1}
Let $K$ be a number field of degree $\geq 3$.
Then there are at most finitely
many three times monogenic orders in $K$.
\end{theorem}

This result is a refinement of work of B\'{e}rczes \cite{Be1}.

The bound $3$ is best possible, i.e., there are number fields $K$
having infinitely many two times monogenic orders.
We believe that if $K$ is an arbitrary number field of degree $\geq 3$,
then with at most finitely many exceptions,
all two times monogenic orders in $K$ are of a special structure.
Below, we state a theorem which confirms this if we impose some restrictions
on $K$.

Let $K$ be a number field of degree at least $3$. An order $\OO$ in $K$
is called of {\bf type I} if there are $\alpha ,\beta\in\OO$
and
$\left(\begin{smallmatrix}a_1&a_2\\a_3&a_4\end{smallmatrix}\right)
\in\GL (2,\Zz )$
such that
\begin{equation}
\label{1.2}
K=\Qq (\alpha ),\ \ \OO =\Zz [\alpha ]=\Zz [\beta ],\ \
\beta =\frac{a_1\alpha +a_2}{a_3\alpha +a_4},\ \ a_3\not= 0.
\end{equation}
Notice that $\beta$ is not $\Zz$-equivalent to $\alpha$, since $a_3\not= 0$
and $K$ has degree at least $3$. So orders of type I are two times
monogenic.

Orders $\OO$ of type II exist only for number fields of degree $4$.
An order $\OO$ in a quartic number field $K$ is called of {\bf type II} if there
are $\alpha ,\beta\in \OO$ and $a_0,a_1,a_2,b_0,b_1,b_2\in\Zz$ 
with $a_0b_0\not= 0$ such that
\begin{eqnarray}\label{1.3}
&&K=\Qq (\alpha ),\ \ \OO =\Zz [\alpha ]=\Zz [\beta ],
\\
\nonumber
&&\beta =a_0\alpha^2 +a_1\alpha +a_2,\ \ \ \alpha =b_0\beta^2 +b_1\beta +b_2.
\end{eqnarray}
Orders of type II are certainly two times monogenic.
At the end of this section, we give examples of number fields 
having infinitely many orders of type I, respectively II.

Let $E$ be a field of characteristic $0$,
and $F=E(\theta )/E$ a finite field extension of degree $d$.
Denote by $\theta^{(1)}\kdots \theta^{(d)}$
the conjugates of $\theta$ over $E$, and by $G$ the normal closure
$E(\theta^{(1)}\kdots \theta^{(d)})$ of $F$ over $E$. We call $F$
\emph{$m$ times transitive} over $E$ ($m\leq d$)
if for any two ordered $m$-tuples of distinct
indices $(i_1\kdots i_m)$, $(j_1\kdots j_m)$ from $\{ 1\kdots d\}$,
there is $\sigma\in{\rm Gal}(G/E)$ such that
\[
\sigma (\theta^{(i_1)})=\theta^{(j_1)}\kdots
\sigma (\theta^{(i_m)})=\theta^{(j_m)}.
\]
If $E=\Qq$, we simply say that $F$ is $m$ times transitive.

We denote by $S_n$ the permutation group on $n$ elements.

Our result on two times monogenic orders is as follows.

\begin{theorem}\label{T1.2}
{\bf (i)}
Let $K$ be a cubic number field. Then every two times monogenic
order in $K$ is of type I.
\\[0.1cm]
{\bf (ii)} Let $K$ be a quartic number field of which the normal
closure has Galois group $S_4$.
Then there are at most finitely many two times monogenic orders in $K$
which are not of type I or of type II.
\\[0.1cm]
{\bf (iii)}
Let $K$ be a four times transitive number field of degree at least $5$.
Then there at most finitely many two times monogenic orders in $K$
which are not of type I.
\end{theorem}

In Section \ref{1a} we present some immediate applications of our results
to canonical number systems.
In Section \ref{2} we formulate generalizations of
Theorems \ref{T1.1} and \ref{T1.2}
for the case that the ground ring
is an arbitrary integrally closed domain which is finitely generated
over $\Z$. Sections \ref{3}--\ref{8} contain auxiliary results,
and Sections \ref{5}--\ref{7}
contain our proofs.

Our proofs of Theorems \ref{T1.1} and \ref{T1.2} use
finiteness results on unit equations
in more than two unknowns, together with some combinatorial
arguments. 
At present, it is not known how to make the results on unit equations
effective,
therefore we are not able
to determine effectively the
three times monogenic orders in Theorem \ref{T1.1}, or the
two times monogenic orders not of type I or II in Theorem \ref{T1.2}.
Although it is possible to estimate from above the number of solutions of unit
equations, it is because of  
the combinatorial arguments in our proofs that we are not able
to estimate from above the numbers of exceptional orders in Theorems \ref{T1.1}
and \ref{T1.2}.

We finish this introduction with constructing number fields
having infinitely many orders of type I, respectively type II.

Let $K$ be a number field of degree $\geq 3$
which is not a totally complex quadratic extension of a totally real field.
By Dirichlet's Unit Theorem, for any proper subfield $L$ of $K$,
the rank of $\OO_L^*$ (the group of units  of the ring of integers of $L$)
is smaller than that of $\OO_K^*$.
We show that $K$ has infinitely many orders of type I.
Take $\left(\begin{smallmatrix}a_1&a_2\\a_3&a_4\end{smallmatrix}\right)\in\GL (2,\Zz )$
with $a_3\not= 0$. Suppose that there is $u_0\in \OO_K^*$
such that $u_0\equiv a_4\pmod{a_3}$.
This is the case for instance if $a_4=1$. By the Euler-Fermat Theorem for
number fields, there is a positive integer $t$ such that $u^t\equiv 1\pmod{a_3}$
for every $u\in \OO_K^*$. Hence
the group of units $u\in \OO_K^*$
with $u\equiv 1\pmod{a_3}$ has finite index in $\OO_K^*$. Consequently,
there are
infinitely many units $u\in \OO_K^*$ with $u\equiv a_4\pmod{a_3}$.
By our assumption on $K$, among these, there are infinitely many  $u$ with
$\Qq (u)=K$. For each such $u$, put
\[
\alpha := \frac{u -a_4}{a_3},\ \ \beta :=\frac{a_1\alpha +a_2}{a_3\alpha +a_4}.
\]
Then clearly, $K=\Qq (\alpha )$.
From the minimal polynomial of $u$ we derive a relation $u^{-1}=f(u)$
with $f\in\Zz [X]$.
Hence $\beta = (a_1\alpha +a_2)f(a_3\alpha +a_4)\in\Zz [\alpha ]$.
Since $\beta =(a_4\beta -a_2)/(-a_3\beta +a_1)$ and
$-a_3\beta +a_1 =\pm u^{-1}$,
we obtain in a similar fashion $\alpha\in\Zz [\beta ]$. Therefore,
$\Zz [\alpha ]=\Zz [\beta ]$. By varying
$\left(\begin{smallmatrix}a_1&a_2\\a_3&a_4\end{smallmatrix}\right)$ and $u$
we obtain infinitely many orders of type I in $K$.

For instance, for $u\in \OO_K^*$ we have $\Zz [u]=\Zz [u^{-1}]$
and the discriminant of this order is the discriminant of
(the minimal polynomial of) $u$.
By Gy\H{o}ry \cite[Corollaire 2.2]{Gy1},
there are at most finitely many units $u\in\OO_K^*$ of given discriminant.
Hence there are infinitely
many distinct orders among $\Zz [u]$ ($u\in\OO_K^*$).

We now construct quartic fields with infinitely many orders of type II.
The construction is based on the theory of cubic resolvents, see
van der Waerden \cite[\S64]{vdWaerden}.

Let $r,s$ be integers such that
the polynomial $f(X)=(X^2-r)^2-X-s$ is irreducible and has Galois group
$S_4$. There are infinitely many such pairs $(r,s)$ (see, e.g.,
Kappe and Warren \cite{KW89}).
Denote by $\alpha^{(1)}=\alpha ,\alpha^{(2)},\alpha^{(3)},\alpha^{(4)}$
the roots of $f$ and let $K:=\Qq (\alpha )$. Define
\begin{eqnarray*}
\eta_1 &:=& -(\alpha^{(1)}+\alpha^{(2)})(\alpha^{(3)}+\alpha^{(4)})
= (\alpha^{(1)}+\alpha^{(2)})^2,
\\
\eta_2 &:=& -(\alpha^{(1)}+\alpha^{(3)})(\alpha^{(2)}+\alpha^{(4)})
= (\alpha^{(1)}+\alpha^{(3)})^2,
\\
\eta_3 &:=& -(\alpha^{(1)}+\alpha^{(4)})(\alpha^{(2)}+\alpha^{(3)})
= (\alpha^{(1)}+\alpha^{(4)})^2.
\end{eqnarray*}
Then
\begin{equation}\label{1.4}
(X-\eta_1)(X-\eta_2)(X-\eta_3)=X^3-4rX^2+4sX-1.
\end{equation}
Take
\[
\sqrt{\eta_1}=\alpha^{(1)}+\alpha^{(2)},\
\sqrt{\eta_2}=\alpha^{(1)}+\alpha^{(3)},\
\sqrt{\eta_3}=\alpha^{(1)}+\alpha^{(4)}.
\]
Then
\begin{equation}\label{1.6}
\sqrt{\eta_1}\cdot \sqrt{\eta_2}\cdot \sqrt{\eta_3} =1.
\end{equation}

By the Gauss-Fermat Theorem over number fields,
there exists a positive integer $t$ such that
\begin{equation}\label{1.5}
\eta_1^t\equiv 1\pmod{4}.
\end{equation}
Consider for $m=0,1,2,\ldots$ the numbers
\begin{eqnarray*}
&&\alpha_m :=\half\left( \sqrt{\eta_1}^{1+2mt}\,+\sqrt{\eta_2}^{1+2mt}\,
+\sqrt{\eta_3}^{1+2mt}\right),
\\
&&\beta_m :=\half\left(\sqrt{\eta_1}^{\,-1-2mt}\,+\sqrt{\eta_2}^{\,-1-2mt}\,
+\sqrt{\eta_3}^{\,-1-2mt}\right).
\end{eqnarray*}
The numbers $\alpha_m$ are invariant under any automorphism that permutes
$\alpha^{(2)},\alpha^{(3)},\alpha^{(4)}$, i.e., under any automorphism
that leaves $K$ invariant, hence they belong to $K$.
Further, they have four distinct conjugates, so $\Qq (\alpha_m)=K$.
Next, by \eqref{1.6},
\[
\beta_m =\alpha_m^2 -r_m,\ \ \alpha_m=\beta_m^2-s_m,
\]
where
\begin{eqnarray*}
&&r_m=\quarter\left(\eta_1^{1+2mt}+\eta_2^{1+2mt}+\eta_3^{1+2mt}\right),
\\
&&s_m=\quarter\left(\eta_1^{-1-2mt}+\eta_2^{-1-2mt}+\eta_3^{-1-2mt}\right).
\end{eqnarray*}
By \eqref{1.4},\eqref{1.5}, $r_m,s_m$ are rational integers, hence
$\alpha_m,\beta_m$ are algebraic integers for every $m$.
We thus obtain for every non-negative integer $m$ an order
$\Zz [\alpha_m]=\Zz [\beta_m]$ of type II in $K$.

We claim that among the orders $\Zz [\alpha_m]$ there
are infinitely many distinct ones.
Denote by $D_m$ the discriminant of $\Zz [\alpha_m]$.
Then $D_m$ is equal to the discriminant of $\alpha_m$,
and a straightforward computation shows that this is equal to the
discriminant of $\eta_1^{1+2mt}$.
By \cite[Corollaire 2.2]{Gy1}, we have
$|D_m|\to\infty$ as $m\to\infty$.
This implies our claim.

\section{Application to canonical number systems}\label{1a}

Let $K$  be an algebraic number field of degree $\geq 2$, and $\OO$ an order in $K$. A nonzero element
$\alpha$ in  $\OO$ is called a {\it basis of a canonical number system} (or CNS basis) for $\OO$
if every nonzero element of $\OO$ can be represented in the form
$$
a_0 + a_1 \alpha + \cdots  + a_m \alpha ^m
$$
with  $m\geq 0$, $a_i \in \{ 0, 1\kdots |N_{K/\Q} (\alpha)|-1 \}$
for $i=0\kdots m$, and $a_m \ne 0$. Canonical number systems can be viewed
as natural generalizations of radix representations of rational integers to algebraic integers.

When there exists a canonical number system in $\OO$, 
then  $\OO$ is called a CNS {\it order}. Orders of
this kind have been intensively investigated; we refer to the survey paper \cite{BrunHuP} and the references
given there.

It was proved by Kov\'acs \cite{KB} and Kov\'acs and Peth\H o \cite{KBP} that $\OO$ is a CNS order if and only
if $\OO$ is monogenic. More precisely, if $\alpha$ is a CNS basis in $\OO$, then it is easily seen that $\OO=\Z[\alpha]$.
Conversely, $\OO=\Z[\alpha]$ does not imply in general that $\alpha$ is a CNS basis. However, in this case
there are infinitely many  $\alpha'$  which are $\Z$-equivalent to $\alpha$ such that $\alpha'$ is a CNS basis for $\OO$.
A characterization of CNS bases in $\OO$ is given in \cite{KBP}.

The close connection between elements $\alpha$ of $\OO$ with $\OO=\Z[\alpha]$ and CNS bases in $\OO$ enables one
to apply results concerning monogenic orders to CNS orders and CNS bases. The results presented in Section \ref{1} have
immediate applications of this type. For example, it follows that up to $\Z$-equivalence there are only finitely many
canonical number systems in $\OO$.

We say that $\OO$ is a $k$-{\it times} CNS order if there are at least $k$ pairwise $\Z$-inequivalent
CNS bases in $\OO$. Theorem 1.1 gives the following.

\begin{corollary}\label{C2.1}
Let $K$ be an algebraic number field of degree $\geq 3$. Then there are at most finitely
many three times CNS orders in $K$.
\end{corollary}

\section{Results over finitely generated domains}\label{2}

\setcounter{equation}{0}

Let $A$ be a domain with quotient field $L$ of characteristic $0$.
Suppose that $A$ is integrally closed, and that $A$ is finitely
generated over $\Z$ as a $\Z$-algebra.
Let $K$ be a finite extension of $L$ of
degree at least $3$, $A_K$ the integral closure of $A$ in $K$,
and $\OO$ an $A$-order in $K$, that is a subring
of $A_K$ which contains $A$ and which has quotient field $K$.
Consider the equation
\begin{equation}\label{2.1}
A[\al]=\OO \ \ \ \text{in} \ \ \al \in \OO.
\end{equation}
The solutions of this equation can be divided into $A$-equivalence
classes, where two elements $\al,\be$ of $\OO$ are called $A$-{\it equivalent}
if $\be = u\al + a$ for some $a \in A$ and $u \in A^*$. Here $A^*$ denotes the
multiplicative group of invertible elements of $A$. As is known (see Roquette \cite{Roquette1}),
$A^*$ is finitely generated.

It was proved by Gy\H ory \cite{Gy17}
that the set of $\al$ with \eqref{2.1} is a
union of finitely many $A$-equivalence classes. An explicit upper bound for the
number of these $A$-equivalence classes has been derived by Evertse and Gy\H ory \cite{EGy1}.
An effective version has been established by Gy\H ory for certain special types of domains
\cite{Gy20}.

We now formulate our generalizations of the results from the previous sections to $A$-orders.
We call an $A$-order $\OO$ $k$ \emph{times monogenic,}
if Eq. \eqref{2.1} has at least $k$ $A$-equivalence
classes of solutions.

\begin{theorem}\label{T2.1}
Let $A$ be a domain with quotient field $L$ of characteristic $0$
which is integrally closed and finitely generated over $\Zz$, and
let $K$ be a finite extension of $L$ of degree $\geq 3$.
Then there are at most finitely many three times monogenic $A$-orders in $K$.
\end{theorem}

We now turn to two times monogenic $A$-orders.
Let again $K$ be a finite extension of $L$
of degree at least $3$. We call $\OO$ an $A$-order in $K$ of {\bf type I} if there
are $\alpha ,\beta\in\OO$ and
$\left(\begin{smallmatrix}a_1&a_2\\a_3&a_4\end{smallmatrix}\right)\in\GL (2,L)$
such that
\begin{equation}\label{2.2}
K=L(\alpha ),\ \OO =A[\alpha ] =A[\beta ],\  \beta =\frac{a_1\alpha +a_2}{a_3\alpha +a_4},\  a_3\not= 0.
\end{equation}
It should be noted that in the previous section (with $L=\Qq,\, A=\Zz$) we had in our definition
\eqref{1.2} of orders of type I
the stronger requirement $\left(\begin{smallmatrix}a_1&a_2\\a_3&a_4\end{smallmatrix}\right)\in\GL (2,\Zz )$
instead of $\left(\begin{smallmatrix}a_1&a_2\\a_3&a_4\end{smallmatrix}\right)\in\GL (2,\Qq )$.
In fact, if $A$ is a principal ideal domain, we can choose $a_1,a_2,a_3,a_4$ in \eqref{2.2} such that
$a_1,a_2,a_3,a_4\in A$ and the ideal generated by $a_1\kdots a_4$ equals $A$.
In that case, according to Lemma \ref{L7.1} proved in Section \ref{8} below, \eqref{2.2} implies that
$\left(\begin{smallmatrix}a_1&a_2\\a_3&a_4\end{smallmatrix}\right)\in\GL (2,A)$.

$A$-orders of type II exist only in extensions of $L$ of degree $4$. Thus, let $K$ be an extension
of $L$ of degree $4$. We call $\OO$ an $A$-order in $K$ of {\bf type II} if there are $\alpha ,\beta\in\OO$
and $a_0,a_1,a_2,b_0,b_1,b_2\in A$ with $a_0b_0\not= 0$, such that
\begin{eqnarray}\label{2.3}
&&K=L(\alpha),\ \ \OO =A[\alpha ]=A[\beta ],
\\
\nonumber
&&\beta =a_0\alpha^2 +a_1\alpha +a_2,\ \ \alpha =b_0\beta^2 +b_1\alpha +b_2.
\end{eqnarray}

\begin{theorem}\label{T2.2}
Let $A$ be a domain with quotient field $L$ of characteristic $0$
which is integrally closed and finitely generated over $\Zz$, and
let $K$ be a finite extension of $L$. Denote by $G$ the normal closure
of $K$ over $L$.
\\[0.1cm]
{\bf (i)} Suppose $[K:L]=3$. Then every two times monogenic $A$-order in $K$ is of type I.
\\[0.1cm]
{\bf (ii)} Suppose $[K:L]=4$ and ${\rm Gal}(G/L)\cong S_4$.
Then there are only finitely many two times monogenic $A$-orders in $K$
which are not of type I or type II.
\\[0.1cm]
{\bf (iii)} Suppose $[K:L]\geq 5$ and that $K$ is four times transitive over $L$.
Then there are only finitely many two times monogenic $A$-orders in $K$
which are not of type I.
\end{theorem}

\section{Equations with unknowns from a finitely generated multiplicative group}\label{3}

\setcounter{equation}{0}

The main tools in the proofs of Theorems \ref{T2.1} and \ref{T2.2}
are finiteness results on polynomial equations of which the unknowns
are taken from finitely generated multiplicative groups.
In this section, we have collected what is needed.
Below, $G$ is a field of characteristic $0$.

\begin{lemma}\label{L1} Let $a_1,a_2 \in G^*$ and let $\Gamma$ be a finitely generated
subgroup of $G^*$.
Then the equation
\begin{equation}\label{3.1}
a_1x_1+a_2x_2=1 \ \ \ \text{in} \ x_1,x_2\in \Gamma
\end{equation}
has only finitely many solutions.
\end{lemma}

\begin{proof}
See Lang \cite{Lang1}.
\end{proof}

A pair $(a_1,a_2)\in (G^*)^2=G^*\times G^*$ is called \emph{normalized}
if $(1,1)$ is a solution to \eqref{3.1}, i.e., $a_1+a_2=1$.
If \eqref{3.1} has a solution,
$(y_1,y_2)$, say, then by replacing $(a_1,a_2)$ by $(a_1y_1,a_2y_2)$
we obtain an equation like
\eqref{3.1} with a normalized pair of coefficients,
whose number of solutions is the same as
that of the original equation.

\begin{lemma}\label{L2}
Let $\Gamma$ be a finitely generated subgroup of $G^*$.
There is a finite set of normalized pairs in $(G^*)^2$,
such that for every normalized pair $(a_1,a_2) \in (G^*)^2$ outside this set, equation \eqref{3.1}
has at most two solutions, the pair $(1,1)$ included.
\end{lemma}

\begin{proof}
This result is due to Evertse, Gy\H ory, Stewart and Tijdeman \cite{EGyST2};
see also \cite{Gy4}. We note that the proof depends ultimately
on the Subspace Theorem, hence it is ineffective.
\end{proof}

We consider more generally polynomial equations
\begin{equation}\label{3.0}
f(x_1\kdots x_n)=0\ \ \mbox{in}\ x_1\kdots x_n\in\Gamma
\end{equation}
where $f$ is a non-zero polynomial from $G[X_1\kdots X_n]$ and $\Gamma$
is a finitely generated subgroup of $G^*$.
Denote by $T$ an auxiliary variable.
A solution $(x_1\kdots x_n)$ of \eqref{3.0} is called \emph{degenerate},
if there are integers $c_1\kdots c_n$, not all zero, such that
\begin{equation}\label{3.00}
f(x_1T^{c_1}\kdots x_nT^{c_n})\equiv 0\ \ \mbox{identically in $T$,}
\end{equation}
and \emph{non-degenerate} otherwise.

\begin{lemma}\label{L0}
Let $f$ be a non-zero polynomial from $G[X_1\kdots X_n]$ and
$\Gamma$ a finitely generated subgroup of $G^*$.
Then Eq. \eqref{3.0} has only finitely many non-degenerate solutions.
\end{lemma}

\begin{proof}
Given a multiplicative abelian group $H$,
we denote by $H^n$ its $n$-fold direct product
with componentwise multiplication.

Let $V$ be the hypersurface given by $f=0$. Notice that the degenerate solutions
${\bf x}$ are precisely those, for which there exists an algebraic subgroup $H$ of $(G^*)^n$
of dimension $\geq 1$ such that ${\bf x}H\subseteq V$.
By a theorem of Laurent \cite{Laurent},
the intersection $V\cap\Gamma^n$ is contained in a finite union of cosets
${\bf x}_1H_1\cup\cdots\cup {\bf x}_rH_r$
where $H_1\kdots H_r$ are algebraic subgroups of $(G^*)^n$,
${\bf x}_1\kdots {\bf x}_r$
are elements of $\Gamma^n$, and ${\bf x}_iH_i\subseteq V$ for $i=1\kdots r$.
The non-degenerate solutions in our lemma are precisely
the zero-dimensional cosets among ${\bf x}_1H_1\kdots{\bf x}_rH_r$,
while the degenerate solutions
are in the union of the positive dimensional cosets.
\end{proof}

\section{Finitely generated domains}\label{4}

\setcounter{equation}{0}
We recall some facts about domains finitely generated over $\Zz$.

Let $A$ be an integrally closed domain
with quotient field $L$ of characteristic $0$ which is finitely generated over $\Z$.
Then $A$ is a Noetherian domain.
Moreover, $A$ is a \emph{Krull domain;}
see e.g. Bourbaki \cite{Bourbaki1}, Chapter VII, \S 1.
This means the following.
Denote by $\PPP(A)$ the collection of minimal non-zero prime ideals of $A$,
these are the non-zero prime ideals that do
not contain a strictly smaller non-zero prime ideal.
Then there exist normalized discrete valuations $\ord_{\fp}$ $(\fp \in \PPP(A))$ on $L$, such that the
following conditions are satisfied:
\begin{eqnarray}
\label{3.2}
&&
\begin{array}{l}
\text{for every $x \in K^*$ there are only finitely many $\fp \in \PPP(A)$ with}
\\
\ord_{\fp}(x) \ne 0,
\end{array}
\\[0.2cm]
\label{3.3}
&&
\, A=\big\{ x \in K \ : \ \ord_{\fp}(x) \ge 0 \ \text{for} \ \fp \in \PPP(A) \big\},
\\[0.2cm]
\label{3.4}
&&\, \fp=\big\{ x \in A \ : \ \ord_{\fp}(x) > 0\big\}\ \mbox{for } \fp\in\PPP(A). \label{3.4}
\end{eqnarray}
These valuations $\ord_{\fp}$ are uniquely determined.
As is easily seen, for $x,y\in L^*$ we have
\begin{equation}\label{3.4a}
\ord_{\fp}(x)=\ord_{\fp}(y)\ \mbox{for all } \fp\in\PPP(A)\Longleftrightarrow
xy^{-1}\in A^* .
\end{equation}

Let $G$ be a finite extension of $L$. Denote by $A_G$ the integral closure of $A$ in $G$,
and by $A_G^*$ the unit group, i.e., group of invertible elements of $A_G$.
We will apply the results from Section \ref{3} with $\Gamma =A_G^*$. To this end,
we need the following lemma.

\begin{lemma}\label{L-finitely-generated}
The group $A_G^*$ is finitely generated.
\end{lemma}

\begin{proof}
The domain $A_G$ is contained in a free $A$-module of rank $[G:L]$.
Since $A$ is Noetherian, the domain $A_G$ is finitely generated as an $A$-module,
and so it is finitely generated as an algebra over $\Zz$.
Then by a theorem of
Roquette \cite{Roquette1}, the group $A_G^*$ is finitely generated.
\end{proof}

\section{Other auxiliary results}\label{8}

\setcounter{equation}{0}

We have collected some elementary lemmas needed in the proofs
of Theorems \ref{T2.1} and \ref{T2.2}.
Let $A$ be an integrally closed domain
with quotient field $L$ of characteristic $0$ which is finitely generated over $\Z$,
and $K$ a finite extension of $L$ with $[K:L]=: d \geq 3$.
Denote by $G$ the normal closure of $K$ over $L$.
Let $\sigma_1=\id, \dots, \sigma_d$
be the distinct $L$-isomorphisms of $K$ in $G$, and for $\al \in K$ write
$\al^{(i)}:=\sigma_i(\al)$ for $i=1, \dots, d$.
Denote by $A_K$ and $A_G$ the integral closures of $A$ in $K$ and $G$, respectively, and
by $A_G^*$ the multiplicative group of invertible elements of $A_G$.

The \emph{discriminant}
of $\al \in K$ is given by
$$
D_{K/L}(\al):=\prod_{1\leq i < j \leq d}\left(\al^{(i)}-\al^{(j)} \right)^2.
$$
This is an element of $L$. We have $L(\alpha )=K$ if and only if
all conjugates of $\alpha$ are distinct, hence if and only if $D_{K/L}(\alpha )\not= 0$.
Further, if $\alpha$ is integral over $A$ then $D_{K/L}(\alpha )\in A$
since $A$ is integrally closed.

\begin{lemma}\label{L2a}
Let $\al, \be\in A_K$ and suppose that $L(\al)=L(\be)=K$, $A[\al]=A[\be]$. Then
\begin{enumerate}
\item $\displaystyle{\frac{\be^{(i)}-\be^{(j)}}{\al^{(i)}-\al^{(j)}} \in A_G^*}$
for $i,j \in \{1, \dots ,d \}$, $i \ne j$,\vskip0.2cm
\item $\displaystyle{\frac{D_{K/L}(\be)}{D_{K/L}(\al)} \in A^*}$.
\end{enumerate}
\end{lemma}

\begin{proof}
\noindent (i) Let $i,j \in \{1, \dots ,d \}$, $i \ne j$. We have $\be = f(\al)$ for some $f\in A[X]$.
Hence
\[
\frac{\be^{(i)}-\be^{(j)}}{\al^{(i)}-\al^{(j)}}=
\frac{f(\al^{(i)})-f(\al^{(j)})}{\al^{(i)}-\al^{(j)}} \in A_G.
\]
Likewise $(\al^{(i)}-\al^{(j)})/(\be^{(i)}-\be^{(j)}) \in A_G$. This proves (i).

\noindent (ii)
We have on the one hand, $D_{K/L}(\be)/D_{K/L}(\al) \in L^*$, on the other hand
$$
\frac{D_{K/L}(\be)}{D_{K/L}(\al)} =
\prod_{1\leq i < j \leq d} \left(\frac{\be^{(i)}-\be^{(j)}}{\al^{(i)}-\al^{(j)}} \right)^2 \in A_G^*.
$$
Since $A$ is integrally closed, this proves (ii).
\end{proof}

We call two elements $\alpha ,\beta$ of $K$ $L$-equivalent if
$\beta =u\alpha +a$ for some $u\in L^* ,a\in L$.

\begin{lemma}\label{L2b}
Let $\alpha ,\beta\in A_K$ and suppose that
$L(\alpha )=L(\beta )=K$, $A[\alpha ]=A[\beta ]$,
and $\alpha ,\beta$ are $L$-equivalent.
Then $\alpha, \beta$ are $A$-equivalent.
\end{lemma}

\begin{proof}
By assumption, $\beta =u\alpha +a$ with $u\in L^*$, $a\in L$.
By the previous lemma, $u^{d(d-1)}=D_{K/L}(\beta )/D_{K/L} (\alpha )\in A^*$,
and then $u\in A^*$ since $A$ is integrally closed. Consequently,
$a=\beta -u\alpha$ is integral over $A$. Hence $a\in A$.
This shows that $\alpha ,\beta$ are $A$-equivalent.
\end{proof}

For $\alpha\in K$ with $ K = L (\alpha )$ we define the ordered
$(d-2)$-tuple
\begin{equation}\label{6.5.taudef}
\tau (\alpha ):=
\Big(\frac{\alpha^{(3)}-\alpha^{(1)}}{\alpha^{(2)}-\alpha^{(1)}}\kdots
\frac{\alpha^{(d)}-\alpha^{(1)}}{\alpha^{(2)}-\alpha^{(1)}}\Big).
\end{equation}

\begin{lemma}\label{L3}
{\bf (i)} Let $\alpha ,\beta$ with $ L (\alpha )= L (\beta )= K$. Then
$\alpha , \beta$ are $ L$-equivalent if and only if $\tau (\alpha
)=\tau (\beta )$.
\\[0.1cm]
{\bf (ii)} Let $\alpha ,\beta\in A_K$ and suppose that
$L(\al)=L(\be)=K$, $A[\al]=A[\be]$.
Then $\alpha , \beta$ are $A$-equivalent if and only if $\tau (\alpha )=\tau
(\beta )$.
\end{lemma}

\begin{proof} (i) If $\alpha ,\beta$ are $ L$-equivalent,
then clearly $\tau (\alpha )=\tau (\beta )$. Assume conversely
that $\tau (\alpha )=\tau (\beta )$. Then there are unique $u\in
G^*$, $a\in G$ such that
\begin{equation}\label{6.5.ua}
(\beta^{(1)},\ldots ,\beta^{(d)})=u(\alpha^{(1)},\ldots
,\alpha^{(d)})+a(1,\ldots ,1).
\end{equation}
In fact, the unicity of $u,a$ follows since thanks to our
assumption $ K = L (\alpha )$, the numbers
$\alpha^{(1)}\kdots\alpha^{(d)}$ are distinct. As for the
existence, observe that \eqref{6.5.ua} is satisfied with
$u=(\beta^{(2)}-\beta^{(1)})/(\alpha^{(2)}-\alpha^{(1)})$,
$a=\beta^{(1)}-u\alpha^{(1)}$.

Take $\sigma$ from the Galois group
$\mbox{Gal}\left( G/L\right)$.
Then $\sigma\circ\sigma_1\kdots\sigma\circ\sigma_d$ is
a permutation of the $L$-isomorphisms
$\sigma_1\kdots\sigma_d:\, K\hookrightarrow G$.
It follows that $\sigma$ permutes
$(\alpha^{(1)},\ldots ,\alpha^{(d)})$ and
$(\beta^{(1)},\ldots ,\beta^{(d)})$ in the same way.
So by applying $\sigma$ to
\eqref{6.5.ua} we obtain
$$
(\beta^{(1)},\ldots ,\beta^{(d)})=\sigma(u)(\alpha^{(1)},\ldots
,\alpha^{(d)})+\sigma(a)(1,\ldots ,1).
$$
By the unicity of $u$,
$a$ in \eqref{6.5.ua} this implies $\sigma(u)=u$, $\sigma(a)=a$.
This holds for every $\sigma\in\mbox{Gal}\left( G/ L\right)$.
So in fact $u\in L^{*}$, $a\in L$, that is, $\alpha$, $\beta$ are
$ L$-equivalent.

(ii) Use Lemma \ref{L2b}.
\end{proof}

We denote by $(a_1\kdots a_r)$ the ideal of $A$ generated by $a_1\kdots a_r$.

\begin{lemma}\label{L7.1}
Let $\alpha ,\beta\in A_K$ with
$L(\alpha )=L(\beta )=K$,
$A[\alpha ]=A[\beta ]$. Suppose there is a matrix
 $\left(\begin{smallmatrix}a_1&a_2\\a_3&a_4\end{smallmatrix}\right)\in\GL (2,L)$
with
\begin{eqnarray}
\label{7.1}
&\displaystyle{\beta =\frac{a_1\alpha +a_2}{a_3\alpha +a_4},\ \ a_3\not= 0,}&
\\
\label{7.2}
&a_1,a_2,a_3,a_4\in A,\ \ (a_1,a_2,a_3,a_4)=(1).&
\end{eqnarray}
Then $\left(\begin{smallmatrix}a_1&a_2\\a_3&a_4\end{smallmatrix}\right)\in\GL (2,A)$.
\end{lemma}

\noindent
{\bf Remark.}
Let $\OO$ be an $A$-order of type I, as defined in Section \ref{2}.
Then there exist $\alpha ,\beta$ with $\OO =A[\alpha ]=A[\beta ]$,
and a matrix
$U:=\left(\begin{smallmatrix}a_1&a_2\\a_3&a_4\end{smallmatrix}\right)\in\GL (2,L)$
with \eqref{7.1}. If $A$ is a principal ideal domain then by taking a suitable
scalar multiple of $U$ we can arrange that \eqref{7.2} also holds, and thus,
that $U\in\GL (2,A)$.

\begin{proof}
Since $\alpha\in A_K$ and $L(\alpha )=K$,
it has a monic minimal polynomial $f\in A[X]$ over $L$ of degree $d$.
Moreover, since $A[\beta ]= A[\alpha ]$, we have
\begin{equation}\label{7.3}
\beta =r_0+r_1\alpha +\cdots +r_{d-1}\alpha^{d-1}\ \mbox{with } r_0\kdots r_{d-1}\in A.
\end{equation}
Hence
\begin{equation}\label{7.4}
(a_3X+a_4)(r_{d-1}X^{d-1}+\cdots +r_0)-a_1X-a_2 =a_3r_{d-1}f(X).
\end{equation}
Equating the coefficients, we see that
\begin{eqnarray}\label{7.5}
&a_4r_0-a_2\in a_3r_{d-1} A,\ \ \
a_4r_1+a_3r_0-a_1\in a_3r_{d-1}A,&
\\
\label{7.6}
&a_4r_j+a_3r_{j-1}\in a_3r_{d-1}A\ (j=2\kdots d-1).&
\end{eqnarray}
We first prove that
\begin{equation}\label{7.7}
a_3^{1-j}r_j\in A\ \ \mbox{for } j=1\kdots d-1.
\end{equation}
In fact, we prove by induction on $i$ ($1\leq i\leq d-1$), the assertion
that $a_3^{1-j}r_j\in A$ for $j=1\kdots i$, and $a_3^{1-i}r_j\in A$
for $j=i+1\kdots d-1$. For $i=1$ this is clear. Let $2\leq i\leq d-1$
and suppose that the assertion is true for $i-1$ instead of $i$.
Then $s_j:= a_3^{2-i}r_j\in A$ for $j=i\kdots d-1$. Further,
by \eqref{7.6}, we have $a_4s_j+a_3s_{j-1}=z_ja_3s_{d-1}$ with $z_j\in A$
for $j=i\kdots d-1$. Next, by \eqref{7.5}
we have $a_1,a_2\in (a_3,a_4)$, and then $(a_3,a_4)=(1)$ by \eqref{7.2}.
That is, there are $x,y\in A$ with $xa_3+ya_4=1$. Consequently,
for $j=i\kdots d-1$, we have
\[
s_j=(xa_3+ya_4)s_j=a_3(xs_j+y(z_js_{d-1}-s_{j-1}))\in a_3A,
\]
i.e., $a_3^{1-i}r_j=a_3^{-1}s_j\in A$.
This completes our induction step, and completes
the proof of \eqref{7.7}.

Now define the binary form $F(X,Y):= Y^df(X/Y)$. Then \eqref{7.4} implies
\[
a_3r_{d-1}F(X,Y)= (a_3X+a_4Y)(\cdots )-Y^{d-1}(a_1X+a_2Y).
\]
Substituting $X=a_4,\,Y=-a_3$, and using \eqref{7.7}, it follows that
\begin{equation}\label{7.8}
F(a_4,-a_3)= s^{-1}(a_1a_4-a_2a_3)\ \mbox{with } s\in A.
\end{equation}

Denote by $\alpha^{(1)}\kdots \alpha^{(d)}$ the conjugates of $\alpha$,
and by $\beta^{(1)}\kdots \beta^{(d)}$ the corresponding conjugates of $\beta$.
Then for the discriminant of $\beta$ we have, by \eqref{7.1}, \eqref{7.8},
\begin{eqnarray*}
&&D_{K/L}(\beta )=\prod_{1\leq i<j\leq d} (\beta^{(i)}-\beta^{(j)})^2
\\
&&\quad =(a_1a_4-a_2a_3)^{d(d-1)}
\left(\prod_{i=1}^d (a_4+a_3\alpha^{(i)})\right)^{-2d+2}\prod_{1\leq i<j\leq d}(\alpha^{(i)}-\alpha^{(j)})^2
\\
&&\quad = (a_1a_4-a_2a_3)^{d(d-1)}F(a_4,-a_3)^{-2d+2}D_{K/L}(\alpha )
\\
&&\quad =
s^{2d-2} (a_1a_4-a_2a_3)^{(d-1)(d-2)}D_{K/L}(\alpha ).
\end{eqnarray*}
On the other hand, by Lemma \ref{L2a}, (ii) we have
$D_{K/L}(\beta )/D_{K/L}(\alpha )\in A^*$. Using also that
$A$ is integrally closed, it follows that
$a_1a_4-a_2a_3\in A^*$.
This completes our proof.
\end{proof}

\section{Proof of Theorem \ref{T2.1}}\label{5}

\setcounter{equation}{0}

The proof splits into two parts.
Consider $\beta\in A_K$ with $K=L(\beta )$.
The first part, which is Lemma \ref{L4} below, implies
that the set of $\beta$ such that
$A[\beta ]$ is three times monogenic, is contained in a
union of at most finitely many $L$-equivalence classes.
The second part, which is Lemma \ref{L5} below, implies that
if $\mathcal{C}$ is a given $L$-equivalence class, then the
set of $\beta\in\mathcal{C}$ such that
$A[\beta ]$ is two times monogenic, is in a union of at most finitely
many $A$-equivalence classes.   
(Lemma \ref{L5} is used 
in the proof of Theorem \ref{T2.2} as well, therefore it deals with two times
monogenic orders.)
Any three times monogenic $A$-order in $K$ can be expressed as $A[\beta ]$.
A combination of Lemmas \ref{L4} and \ref{L5} clearly yields that the set
of such $\beta$ lies in finitely
many $A$-equivalence classes. Since $A$-equivalent $\beta$
give rise to equal $A$-orders $A[\beta ]$, there are only
finitely many three times monogenic orders in $K$.

\begin{lemma}\label{L4}
The set of $\beta$ such that
\begin{equation}\label{6.6.4}
\beta\in A_{ K},\  L (\beta )= K,\ \  A[\beta ]\ \mbox{is three times monogenic}
\end{equation}
is contained in a union of at most finitely many $L$-equivalence classes.
\end{lemma}

\begin{proof}
Assume the contrary. Then there is an infinite sequence of triples
$\{ (\beta_{1p},\beta_{2p},\beta_{3p}):\, p=1,2,\ldots\}$
such that
\begin{equation}\label{5.101}
\beta_{hp}\in A_K,\ L(\beta_{hp})=K\ \mbox{for } h=1,2,3,\ p=1,2,\ldots ;
\end{equation}
\begin{equation}\label{5.102}
\beta_{1p}\ (p=1,2,\ldots )\ \mbox{lie in different $L$-equivalence classes}
\end{equation}
and for $p=1,2,\ldots$ ,
\begin{equation}\label{5.103}
\left\{
\begin{array}{l}
A[\beta_{1p}]=A[\beta_{2p}]=A[\beta_{3p}],\\
\mbox{$\beta_{1p},\beta_{2p},\beta_{3p}$ lie in different
$A$-equivalence classes}
\end{array}\right.
\end{equation}
(so the $\beta_{1p}$ play the role of $\beta$ in the statement of our lemma).
For any three distinct indices $i,j,k$ from $\{ 1\kdots d\}$, and for $h=1,2,3$,
$p=1,2,\ldots$, put
\[
\beta_{hp}^{(ijk)}:=\frac{\beta_{hp}^{(i)}-\beta_{hp}^{(j)}}
{\beta_{hp}^{(i)}-\beta_{hp}^{(k)}}.
\]
By \eqref{5.101}, these numbers are well-defined and non-zero.

We start with some observations.
Let $i,j,k$ be any three
distinct indices from $\{ 1\kdots d\}$.
By Lemma \ref{L2a} and the obvious identities
$\beta_{hp}^{(ijk)}+\beta_{hp}^{(kji)}=1$ ($h=1,2,3$), the pairs
$(\beta_{hp}^{(ijk)}/\beta_{1p}^{(ijk)},\beta_{hp}^{(kji)}/\beta_{1p}^{(kji)})$ ($h=1,2,3$)
are solutions to
\begin{equation}\label{5.104}
\beta_{1p}^{(ijk)}x\, + \beta_{1p}^{(kji)}y\, =1\ \ \mbox{in } x,y\in A_G^*.
\end{equation}
Notice that \eqref{5.104} has solution $(1,1)$. So according to
Lemmas \ref{L2}, \ref{L-finitely-generated}, there is a finite set $\mathcal{A}_{ijk}$ such that
if $\beta_{1p}^{(ijk)}\not\in\mathcal{A}_{ijk}$, then \eqref{5.104}
has at most two solutions, including $(1,1)$.
In particular,
there are at most two distinct pairs among
$(\beta_{hp}^{(ijk)}/\beta_{1p}^{(ijk)},\beta_{hp}^{(kji)}/\beta_{1p}^{(kji)})$
$(h=1,2,3)$. Consequently,
\begin{equation}\label{5.108}
\beta_{1p}^{(ijk)}\not\in\mathcal{A}_{ijk}\Longrightarrow
\mbox{two among $\beta_{1p}^{(ijk)},\beta_{2p}^{(ijk)},\beta_{3p}^{(ijk)}$ are equal.}
\end{equation}

We start with the case $d=3$. Then $\tau (\beta_{hp})=(\beta_{hp}^{(132)})$
for $h=1,2,3$. By \eqref{5.102} and Lemma \ref{L3},(i) the numbers
$\beta_{1p}^{(132)}$ ($p=1,2,\ldots$) are pairwise distinct.
By \eqref{5.108} and Lemma \ref{L3},(ii), for all but finitely many $p$, two among the
numbers $\beta_{hp}^{(132)}$ ($h=1,2,3$) are equal and hence
two among $\beta_{hp}$ ($h=1,2,3$) are $A$-equivalent which contradicts
\eqref{5.103}.

Now assume $d\geq 4$.
We have to distinguish between subsets $\{i,j,k\}$
of $\{ 1\kdots d\}$
and indices $h$ for which there are infinitely many distinct
numbers among $\beta_{hp}^{(ijk)}$ ($p=1,2,\ldots$),
and $\{i,j,k\}$ and $h$ for which among these numbers there are only finitely many
distinct ones.
This does not depend on the choice of ordering of $i,j,k$, since any permutation
of $(i,j,k)$ transforms $\beta_{hp}^{(ijk)}$ into one of
$(\beta_{hp}^{(ijk)})^{-1}$,
$1-\beta_{hp}^{(ijk)}$, $(1-\beta_{hp}^{(ijk)})^{-1}$,
$1-(\beta_{hp}^{(ijk)})^{-1}$, $(1-(\beta_{hp}^{(ijk)})^{-1})^{-1}$.

There is a subset $\{i,j,k\}$ of $\{ 1\kdots d\}$
such that there are infinitely many distinct numbers among
$\beta_{1p}^{(ijk)}$ ($p=1,2,\ldots$).
Indeed, if this were not the case, then there would be only finitely many
distinct tuples among
$\tau (\beta_{1p})=(\beta_{1p}^{(132)}\kdots \beta_{1p}^{(1d2)})$,
and then by Lemma \ref{L3},(i) the numbers $\beta_{1p}$ would lie in
only finitely many $L$-equivalence classes, contradicting \eqref{5.102}.
There is an infinite subsequence of indices $p$
such that the numbers $\beta_{1p}^{(ijk)}$ are pairwise distinct.
Suppose there is another subset $\{i',j',k'\}\not=\{ i,j,k\}$
such that if $p$ runs through the infinite subsequence just chosen,
then $\beta_{1p}^{(i'j'k')}$ runs through an infinite set. Then
for some infinite subsequence of these $p$,
the numbers $\beta_{1p}^{(i'j'k')}$ are pairwise distinct.
Continuing in this way, we infer that there is a non-empty collection $\mathcal{S}$
of $3$-element subsets $\{i,j,k\}$ of $\{ 1\kdots d\}$,
and an infinite sequence $\mathcal{P}$ of indices $p$,
such that for each $\{i,j,k\}\in\mathcal{S}$ the numbers $\beta_{1p}^{(ijk)}$
($p\in\mathcal{P}$)
are pairwise distinct, while for each $\{i,j,k\}\not\in\mathcal{S}$,
there are only
finitely many distinct elements among $\beta_{1p}^{(ijk)}$ ($p\in\mathcal{P}$).

Notice that if $\{ i,j,k\}\not\in\mathcal{S}$, then
among the equations \eqref{5.104} with $p\in\mathcal{P}$,
there are only finitely
many distinct ones,
and by Lemmas \ref{L1}, \ref{L-finitely-generated}, each of these
equations has only finitely many solutions.
Therefore, there are only finitely many distinct numbers among
$\beta_{hp}^{(ijk)}/\beta_{1p}^{(ijk)}$ hence only finitely many among
$\beta_{hp}^{(ijk)}$ ($h=2,3$, $p\in\mathcal{P}$).
Conversely, if $\{i,j,k\}\in\mathcal{S}$, $h\in\{ 2,3\}$, there are
infinitely many distinct numbers among $\beta_{hp}^{(ijk)}$
($p\in\mathcal{P}$). For if not, then by the same argument,
interchanging the roles of $\beta_{hp}$, $\beta_{1p}$, it would follow
that there are only finitely many distinct numbers among
$\beta_{1p}^{(ijk)}$ ($p\in\mathcal{P}$),
contradicting $\{i,j,k\}\in\mathcal{S}$.

We conclude that there is an infinite subsequence of $p$,
which after renaming we may assume to be $1,2,\ldots$, such that
for $h=1,2,3$,
\begin{equation}\label{5.106}
\beta_{hp}^{(ijk)}\ (p=1,2,\ldots )\ 
\mbox{are pairwise distinct if $\{ i,j,k\}\in\mathcal{S}$,}
\end{equation}
\begin{equation}\label{5.107}
\begin{array}{l}
\mbox{there are only finitely many distinct numbers among}
\\
\beta_{hp}^{(ijk)}\ (p=1,2,\ldots )\ \mbox{if } \{i,j,k\}\not\in\mathcal{S}.
\end{array}
\end{equation}
Notice that this characterization of $\mathcal{S}$ is symmetric
in $\beta_{hp}$ ($h=1,2,3$);
this will be used frequently.

We frequently use the following property of $\mathcal{S}$:
if $i,j,k,l$ are any four distinct indices from $\{ 1\kdots d\}$,
then
\begin{equation}\label{5.107b}
\{i,j,k\}\in\mathcal{S}\Longrightarrow \{i,j,l\}\in\mathcal{S}\ \mbox{or }
\{i,k,l\}\in\mathcal{S}.
\end{equation}
Indeed, if $\{i,j,l\},\{i,k,l\}\not\in\mathcal{S}$
then also $\{i,j,k\}\not\in\mathcal{S}$ since
$\beta_{hp}^{(ijk)}=\beta_{hp}^{(ijl)}/\beta_{hp}^{(ikl)}$.

Pick a set from $\mathcal{S}$, which without loss of generality
we may assume to be $\{1,2,3\}$. By \eqref{5.107b},
for $k=4\kdots d$ at least one of the sets $\{1,2,k\}$, $\{1,3,k\}$
belongs to $\mathcal{S}$. Define the set of pairs
\begin{equation}\label{5.109}
\mathcal{C}:=\Big\{ (j,k):\, j\in\{ 2,3\},\, k\in\{ 3\kdots d\},\, j\not= k ,\, \{1,j,k\}\in\mathcal{S}\Big\}.
\end{equation}
Thus, for each $k\in\{ 3\kdots d\}$ there is $j$ with $(j,k)\in\mathcal{C}$.
Further, for every $p=1,2,\ldots$ there is a pair $(j,k)\in\mathcal{C}$ such that
\[
\beta_{1p}^{(1jk)}\not=\beta_{2p}^{(1jk)}.
\]
Indeed, if this were not the case, then since
$\beta_{hp}^{(12k)}=\beta_{hp}^{(13k)}\beta_{hp}^{(123)}$,
it would follow that for some $p$,
\[
\beta_{1p}^{(12k)}=\beta_{2p}^{(12k)}\ \mbox{for } k=3\kdots d,
\]
and then $\tau (\beta_{1p})=\tau (\beta_{2p})$.
Together with Lemma \ref{L3},(ii) this would imply that
$\beta_{1p}$, $\beta_{2p}$ are $A$-equivalent, contrary to
\eqref{5.103}. Clearly, there is a pair $(j,k)\in\mathcal{C}$
such that $\beta_{1p}^{(1jk)}\not=\beta_{2p}^{(1jk)}$ for infinitely many $p$.
After interchanging the indices $2$ and $3$ if $j=3$  and then permuting
the indices $3\kdots d$, which does not affect the above argument, 
we may assume that $j=2,k=3$.
That is, we may assume that $\{1,2,3\}\in\mathcal{S}$ and
\[
\beta_{1p}^{(123)}\not=\beta_{2p}^{(123)}\ \mbox{for infinitely many $p$.}
\]
We now bring \eqref{5.108} into play. It implies that for
infinitely many $p$ we have
$\beta_{3p}^{(123)}\in\{ \beta_{1p}^{(123)},\beta_{2p}^{(123)}\}$.
After interchanging $\beta_{1p}$, $\beta_{2p}$
(which does not affect the definition of $\mathcal{S}$ or the above
arguments) we may assume that $\{1,2,3\}\in\mathcal{S}$
and
\begin{equation}\label{5.110}
\beta_{1p}^{(123)}=\beta_{3p}^{(123)}\not=\beta_{2p}^{(123)}
\end{equation}
for infinitely many $p$.

We repeat the above argument.
After renaming again, we may assume that the above infinite sequence of indices
$p$ for which \eqref{5.110} is true is $p=1,2,\ldots\,$,
and thus, \eqref{5.106} and \eqref{5.107} are true again.
Define again the set
$\mathcal{C}$ by \eqref{5.109}.
Similarly as above, we conclude that there is a pair $(j,k)\in\mathcal{C}$
such that among $p=1,2,\ldots$
there is an infinite subset with
$\beta_{1p}^{(1jk)}\not=\beta_{3p}^{(1jk)}$. Then necessarily, $k\not= 3$.
After interchanging $2$ and $3$ if $j=3$ (which does not affect
\eqref{5.110}) and rearranging the other indices $4\kdots d$, we may assume
that $j=2$, $k=4$. Thus, $\{1,2,3\},\{1,2,4\}\in\mathcal{S}$
and there are infinitely many $p$ for which we have
\eqref{5.110} and
\[
\beta_{1p}^{(124)}\not=\beta_{3p}^{(124)}.
\]
By \eqref{5.108}, for all but finitely many of these $p$ we have
$\beta_{2p}^{(124)}\in\{ \beta_{1p}^{(124)},\beta_{3p}^{(124)}\}$.
After interchanging $\beta_{1p},\beta_{3p}$ if necessary,
which does not affect \eqref{5.110}, we may conclude that
$\{1,2,3\},\{1,2,4\}\in\mathcal{S}$ and there are infinitely many $p$ with \eqref{5.110}
and
\begin{equation}\label{5.111}
\beta_{1p}^{(124)}=\beta_{2p}^{(124)}\not= \beta_{3p}^{(124)}.
\end{equation}

Next, by \eqref{5.107b}, at least one of $\{1,3,4\}$, $\{2,3,4\}$ belongs to $\mathcal{S}$.
Relations \eqref{5.110}, \eqref{5.111} remain unaffected if we interchange
$\beta_{hp}^{(1)}$ and $\beta_{hp}^{(2)}$,
so without loss of generality, we may assume that $\{1,3,4\}\in\mathcal{S}$.
By \eqref{5.108}, for all but finitely many of the $p$ with
\eqref{5.110} and \eqref{5.111}, at least two among the numbers
$\beta_{hp}^{(134)}$ ($h=1,2,3$) must be equal. Using \eqref{5.110}, \eqref{5.111}
and $\beta_{hp}^{(134)}=\beta_{hp}^{(124)}/\beta_{hp}^{(123)}$,
it follows that $\{1,2,3\},\{1,2,4\},\{1,3,4\}\in\mathcal{S}$ and for infinitely many $p$
we have \eqref{5.110},\eqref{5.111} and
\begin{equation}\label{5.112}
\beta_{2p}^{(134)}=\beta_{3p}^{(134)}\not= \beta_{1p}^{(134)}.
\end{equation}

We now show that this is impossible. For convenience we introduce the notation
\[
\tilde{\beta}_{hp}^{(i)}:=
\frac{\beta_{hp}^{(i)}-\beta_{hp}^{(4)}}{\beta_{hp}^{(3)}-\beta_{hp}^{(4)}} =\beta_{hp}^{(4i3)}
\]
for $h=1,2,3$, $i=1,2,3,4$, $p=1,2,\ldots$. Notice that $\tilde{\beta}_{hp}^{(3)}=1$,
$\tilde{\beta}_{hp}^{(4)}=0$, and
$\beta_{hp}^{(ijk)}=
\frac{\tilde{\beta}_{hp}^{(i)}-\tilde{\beta}_{hp}^{(j)}}{\tilde{\beta}_{hp}^{(i)}-\tilde{\beta}_{hp}^{(k)}}$
for any distinct $i,j,k\in\{ 1,2,3,4\}$. Thus,
\eqref{5.110}--\eqref{5.112} translate into
\begin{eqnarray}
\label{5.113}
&&\frac{\tilde{\beta}_{1p}^{(1)}-
\tilde{\beta}_{1p}^{(2)}}{\tilde{\beta}_{1p}^{(1)}-1}=
\frac{\tilde{\beta}_{3p}^{(1)}-
\tilde{\beta}_{3p}^{(2)}}{\tilde{\beta}_{3p}^{(1)}-1}\not=
\frac{\tilde{\beta}_{2p}^{(1)}
-\tilde{\beta}_{2p}^{(2)}}{\tilde{\beta}_{2p}^{(1)}-1},
\\
\label{5.114}
&&\frac{\tilde{\beta}_{1p}^{(1)}-
\tilde{\beta}_{1p}^{(2)}}{\tilde{\beta}_{1p}^{(1)}}=
\frac{\tilde{\beta}_{2p}^{(1)}-
\tilde{\beta}_{2p}^{(2)}}{\tilde{\beta}_{2p}^{(1)}}\not=
\frac{\tilde{\beta}_{3p}^{(1)}-
\tilde{\beta}_{3p}^{(2)}}{\tilde{\beta}_{3p}^{(1)}},
\\
\label{5.115}
&&\frac{\tilde{\beta}_{2p}^{(1)}-1}{\tilde{\beta}_{2p}^{(1)}}=
\frac{\tilde{\beta}_{3p}^{(1)}-1}{\tilde{\beta}_{3p}^{(1)}}\not=
\frac{\tilde{\beta}_{1p}^{(1)}-1}{\tilde{\beta}_{1p}^{(1)}}.
\end{eqnarray}

We distinguish between the cases $\{ 2,3,4\}\in\mathcal{S}$
and $\{ 2,3,4\}\not\in\mathcal{S}$.

First suppose that $\{ 2,3,4\}\in\mathcal{S}$.
Then by \eqref{5.108}, there are infinitely many $p$ such that
\eqref{5.113}--\eqref{5.115} hold and at least two among
$\tilde{\beta}_{hp}^{(2)}=\beta_{hp}^{(423)}$
($h=1,2,3$) are equal.
But this is impossible, since \eqref{5.113},\eqref{5.114} imply $\tilde{\beta}_{1p}^{(2)}\not=\tilde{\beta}_{2p}^{(2)}$;
\eqref{5.113},\eqref{5.115} imply $\tilde{\beta}_{1p}^{(2)}\not=\tilde{\beta}_{3p}^{(2)}$;
and \eqref{5.114},\eqref{5.115} imply $\tilde{\beta}_{2p}^{(2)}\not=\tilde{\beta}_{3p}^{(2)}$.

Hence $\{ 2,3,4\}\not\in\mathcal{S}$.
This means that there are only finitely many distinct numbers among
$\tilde{\beta}_{hp}^{(2)}=\beta_{hp}^{(423)}$, ($h=1,2,3,\, p=1,2,\ldots$).
It follows that there are (necessarily non-zero) constants $c_1,c_2,c_3$
such that $\tilde{\beta}_{hp}^{(2)}=c_h$ for $h=1,2,3$ and infinitely many $p$.
By \eqref{5.115}, \eqref{5.114}, respectively, we have for all these $p$
that $\tilde{\beta}_{2p}^{(1)}=\tilde{\beta}_{3p}^{(1)}$ and
$\tilde{\beta}_{2p}^{(1)}=(c_2/c_1)\tilde{\beta}_{1p}^{(1)}$.
By substituting this into \eqref{5.113}, we get
\[
\frac{\tilde{\beta}_{1p}^{(1)}-c_1}{\tilde{\beta}_{1p}^{(1)}-1}=
\frac{c_2\tilde{\beta}_{1p}^{(1)}-c_1c_3}{c_2\tilde{\beta}_{1p}^{(1)}-c_1}.
\]
By \eqref{5.113}, \eqref{5.115} we have $c_1\not=c_3$,
hence
\[
\tilde{\beta}^{(1)}_{1p}=\beta_{1p}^{(413)}=\frac{c_1(c_1-c_3)}{c_1c_2+c_1-c_2-c_1c_3}
\]
is a constant independent of $p$.
But this contradicts $\{1,3,4\}\in\mathcal{S}$ and \eqref{5.106}.

So our assumption that Lemma \ref{L4} is false
leads in all cases to a contradiction.
This completes our proof.
\end{proof}

\begin{lemma}\label{L5}
Let $\mathcal{C}$ be an $ L$-equivalence class in $K$. Then the
set of $\beta$ such that
\begin{equation}\label{5.200}
\beta\in A_{ K}\cap\mathcal{C},\ L(\beta )=K,\ \ \ A[\beta ]\ \mbox{is two times monogenic}
\end{equation}
is contained in a union of at most finitely many $A$-equivalence classes.
\end{lemma}

\noindent
{\bf Remark.} As mentioned before,
Lemma \ref{L5} is used also in the proof of Theorem \ref{T2.2}.
Our proof of Lemma \ref{L5} does not enable to estimate the
number of $A$-equivalence classes.
It is for this reason that we can not prove
quantitative versions of Theorems \ref{T2.1} and \ref{T2.2}.

\begin{proof}
We assume that the set of $\beta$ with \eqref{5.200} is not
contained in a union of finitely many $A$-equivalence classes
and derive a contradiction.

Pick $\beta$ with \eqref{5.200}. Then there exist numbers $\alpha$ such that
$A[\alpha]=A[\beta]$ and $\alpha$ is not $A$-equivalent to
$\beta$. Consider such $\alpha$.
Then from the identities
\[
\frac{\alpha^{(i)}-\alpha^{(1)}}{\alpha^{(2)}-\alpha^{(1)}}
+\frac{\alpha^{(2)}-\alpha^{(i)}}{\alpha^{(2)}-\alpha^{(1)}}=1\ \ (i=3\kdots d)
\]
and Lemma \ref{L2a} it follows
that the pairs
\begin{equation}\label{6.6.12}
\left(
\frac{\alpha^{(i)}-\alpha^{(1)}}{\alpha^{(2)}-\alpha^{(1)}},
\frac{\alpha^{(2)}-\alpha^{(i)}}{\alpha^{(2)}-\alpha^{(1)}}
\right)\ \ (i=3\kdots d)
\end{equation}
satisfy
\[
x+y=1\ \ \ \mbox{in $x,y\in\Gamma$,}
\]
where $\Gamma$ is the multiplicative group generated by
$A_G^*$ and the numbers
\[
\frac{\beta^{(i)}-\beta^{(1)}}{\beta^{(2)}-\beta^{(1)}},\
\frac{\beta^{(2)}-\beta^{(i)}}{\beta^{(2)}-\beta^{(1)}}\ \
(i=3\kdots d).
\]
By Lemma \ref{L3},(i), the group $\Gamma$ depends only on the
given $ L$-equivalence class $\mathcal{C}$ and is otherwise
independent of $\beta$. By Lemma \ref{L-finitely-generated}, the group
$\Gamma$ is finitely generated, and then by Lemma \ref{L1}, the pairs
(\ref{6.6.12}) belong to a finite set depending only on
$\Gamma$, hence only on $\mathcal{C}$.
Therefore, the tuple $\tau (\alpha )$ belongs to a
finite set depending only on $\mathcal{C}$. In view of Lemma
\ref{L3},(i), this means that $\alpha$ belongs to a union of
finitely many $ L$-equivalence classes which depends on
$\mathcal{C}$ but is otherwise independent of $\beta$. Now by
Dirichlet's box principle,
there is an $ L$-equivalence class
$\mathcal{C}'$ with the following property: the set of $\beta$
such that
\begin{equation}\label{6.6.13}
\left\{\begin{array}{l} \beta\in A_{ K},\  L (\beta )= K,\
\beta\in\mathcal{C},\\[0.25cm]
\mbox{there is } \alpha\in\mathcal{C}' \mbox{ such that }
A[\alpha]=A[\beta]
\\[0.05cm]
\mbox{and } \alpha \mbox{ is not } A\mbox{-equivalent to } \beta
\end{array}\right.
\end{equation}
cannot be contained in a union of finitely many $A$-equivalence
classes.

Fix $\beta_0$ with \eqref{6.6.13} and then fix $\alpha_0$ such
that $A[\alpha_0 ]=A[\beta_0]$, $\alpha_0\in\mathcal{C}'$ and
$\alpha_0$ is not $A$-equivalent to $\beta_0$.

Let $\beta$ be an arbitrary number with \eqref{6.6.13}. Choose
$\alpha$ such that $A[\alpha ]=A[\beta ]$, $\alpha\in\mathcal{C}'$
and $\alpha$ is not $A$-equivalent to $\beta$. Then there are
$u,u'\in L^*$, $a,a'\in L$ with
\begin{equation}\label{6.6.14}
\beta=u\beta_0+a,\,\, \alpha=u'\alpha_0+a' .
\end{equation}
For these $u,u'$ we have
\begin{equation}\label{3.5}
D_{K/L}(\beta)= u^{d(d-1)}D_{K/L}(\beta_0),\ \
D_{K/L}(\alpha)=u'^{d(d-1)}D_{K/L}(\alpha_0).
\end{equation}
On the other hand, it follows from $A[\alpha_0 ]=A[\beta_0 ]$,
$A[\alpha ]=A[\beta ]$ and Lemma \ref{L2a} (ii) that
$D_{K/L}(\be)/D_{K/L}(\al)\in A^*$ and $D_{K/L}(\be_0)/D_{K/L}(\al_0)\in A^*$.
Combined with \eqref{3.5} and our assumption that $A$ is integrally closed, this gives
\begin{equation}\label{6.6.15}
u'/u\in A^*.
\end{equation}
Since $ L (\beta_0)= K$ and $\alpha_0\in A[\beta_0]$ there is a
unique polynomial $F_0\in L[X]$ of degree $<d$, which in fact
belongs to $A[X]$, such that $\alpha_0=F_0(\beta_0)$. Likewise,
there is a unique polynomial $F\in L[X]$ of degree $<d$ which in
fact belongs to $A[X]$, such that  $\alpha=F(\beta)$. Inserting
(\ref{6.6.14}), it follows that
$F(X)=u'F_0\left((X-a)/u\right)+a'$. Suppose that
$F_0=\sum_{j=0}^m a_jX^j$ with $m<d$ and $a_m\not= 0$. Then $F$
has leading coefficient $a_mu'u^{-m}$ which belongs to $A$.
Together with \eqref{6.6.15} this implies
\begin{equation}\label{6.6.17}
u^{1-m}a_m\in A.
\end{equation}
Further, by \eqref{3.5}
\begin{equation}\label{6.6.18}
u^{d(d-1)}D_{K/L}(\beta_0)=D_{K/L}(\beta)\in A.
\end{equation}

We distinguish between the cases $m>1$ and $m=1$. First let $m>1$.
We have shown that every $\beta$ with \eqref{6.6.13} can be
expressed as $\beta =u\beta_0 +a$ with $u\in  L^*$, $a\in L$ and
moreover, $u$ satisfies \eqref{6.6.17}, \eqref{6.6.18}. Hence
\[
-\frac{\ord_{\fp}(D_{K/L} (\beta_0))}{d(d-1)}\leq\ord_{\fp}(u) \leq
\frac{\ord_{\fp}(a_m)}{m-1}\ \mbox{for } \fp\in\PPP (A),
\]
where $\PPP (A)$ is the collection of minimal non-zero prime ideals of $A$
and $\ord_{\fp}$ ($\fp\in\PPP (A)\,$) are the associated discrete valuations,
as explained in Section \ref{4}.
Thus, for the tuple $v(u):= (\ord_{\fp}(u):\, \fp\in\PPP (A)\,)$ we have only
finitely many possibilities.

We partition the set of $\beta$ with \eqref{6.6.13} into a finite number of 
classes according to the tuple $v(u)$.
Let $\beta_1=u_1\beta_0+a_1$, $\beta_2=u_2\beta_0+a_2$ belong to the same
class, where $u_1,u_2\in L^*$ and $a_1,a_2\in L$.
Then $v(u_1)=v(u_2)$ and so, $u_1u_2^{-1}\in A^*$ by \eqref{3.4a}.
Hence $\beta_2=v\beta_1+b$ with $v\in A^*$ and
$b\in  L$. But $b=\beta_2-v\beta_1$ is integral over $A$, hence
belongs to $A$ since $A$ is integrally closed. So two elements
with \eqref{6.6.13} belonging to the same class are
$A$-equivalent. But then, the set of $\beta$ with (\ref{6.6.13})
is contained in a union of finitely many $A$-equivalence classes,
which is against our assumption.

Now assume that $m=1$. Then
$$
\alpha_0=a_1\beta_0+a_0\quad\mbox{ with } a_1\in A\setminus\{
0\},\, a_0\in A,
$$
hence $a_1^{d(d-1)}=D_{K/L} (\alpha_0)/D_{K/L} (\beta_0)$.
By Lemma \ref{L2a} (ii)
we have $a_1^{d(d-1)}\in A^*$,
and then $a_1\in A^*$ since by assumption $A$ is integrally
closed. Hence $\alpha_0$, $\beta_0$ are $A$-equivalent, which is
against our choice of $\alpha_0$, $\beta_0$. We arrive again at a
contradiction.

Consequently, our initial assumption that the set of $\beta$
with \eqref{5.200}

cannot be contained in finitely many $A$-equivalence classes leads
to a contradiction. This proves Lemma \ref{L5}.
\end{proof}

Now our proof of Theorem \ref{T2.1} is complete.

\section{Reduction of Theorem \ref{T2.2} to a polynomial unit equation}\label{6}
\setcounter{equation}{0}

We keep the assumptions and notation from the previous sections.
In particular, $A$ is an integrally closed domain with quotient field $L$ of characteristic $0$
which is finitely generated
over $\Zz$ and $K$ is a finite extension of $L$.
Further, we denote by $G$ the normal closure of $K$ over $L$.
As it will turn out, the proof of part (i) of Theorem \ref{T2.2} is elementary.
Therefore, in this section we assume that $[K:L]=:d\geq 4$.
Let $\OO = A[\alpha ]=A[\beta ]$ be a two times monogenic $A$-order in $K$,
where $\alpha ,\beta$ are not $A$-equivalent.

By Lemma \ref{L2a},(i) we have
\begin{equation}\label{6.6.1}
\varepsilon_{ij}:=
\frac{\alpha^{(i)}-\alpha^{(j)}}{\beta^{(i)}-\beta^{(j)}} \in
A_G^*\ \mbox{for } i,j=1\kdots d,\, i\not=j,
\end{equation}
where $A_G^*$ is the unit group of the integral closure of $A$ in $G$.
Let $i,j,k$ be any three distinct indices from $\{ 1\kdots d\}$.
By Lemma \ref{L2a}, the identity
\[
\frac{\beta^{(j)}-\beta^{(i)}}{\beta^{(j)}-\beta^{(k)}}
+
\frac{\beta^{(i)}-\beta^{(k)}}{\beta^{(j)}-\beta^{(k)}}=1
\]
and a similar identity for $\alpha$, the two pairs
$(1,1)$ and $(\ve_{ij}/\ve_{jk},\ve_{ik}/\ve_{jk})$
satisfy
\begin{equation}\label{6.6.3}
\frac{\beta^{(j)}-\beta^{(i)}}{\beta^{(j)}-\beta^{(k)}}\cdot
x\,+\,
\frac{\beta^{(i)}-\beta^{(k)}}{\beta^{(j)}-\beta^{(k)}}\cdot y=1\
\ \mbox{ in }x, y\in A_G^* .
\end{equation}
Now a straightforward computation gives
\begin{equation}\label{6.1}
\frac{\ve_{ik}}{\ve_{jk}}-1\, =\,
\frac{\beta^{(i)}-\beta^{(j)}}{\beta^{(i)}-\beta^{(k)}}\cdot
\left(\frac{\ve_{ij}}{\ve_{jk}}-1\right).
\end{equation}
This is valid for any three distinct indices $i,j,k$.
Now take four distinct indices $i,j,k,l$ from $\{ 1\kdots d\}$.
By applying \eqref{6.1}
but with the respective triples $(i,j,k)$, $(i,k,l)$, $(i,l,j)$
replacing $(i,j,k)$, and taking the product, the terms with the
conjugates of $\beta$ disappear, and we obtain
\begin{eqnarray}\label{6.2}
&&\left(\frac{\ve_{ik}}{\ve_{jk}}-1\right)
\left(\frac{\ve_{il}}{\ve_{kl}}-1\right)
\left(\frac{\ve_{ij}}{\ve_{jl}}-1\right)
\\
\nonumber
&&\qquad\qquad\qquad\qquad =\,
\left(\frac{\ve_{ij}}{\ve_{jk}}-1\right)
\left(\frac{\ve_{ik}}{\ve_{kl}}-1\right)
\left(\frac{\ve_{il}}{\ve_{jl}}-1\right).
\end{eqnarray}
In the remainder of this section we focus on the equation
\begin{eqnarray}\label{6.4}
&&(x_1-1)(x_2-1)(x_3-1)=(y_1-1)(y_2-1)(y_3-1)\\
\nonumber
&&\qquad\qquad\qquad\qquad\qquad\qquad\qquad
\mbox{in } x_1,x_2,x_3,y_1,y_2,y_3\in\Gamma
\end{eqnarray}
where $\Gamma$ is a finitely generated multiplicative group,
contained in a field of characteristic $0$. As we just observed,
the tuple
\begin{equation}\label{6.3}
\left(\frac{\ve_{ik}}{\ve_{jk}}, \frac{\ve_{il}}{\ve_{kl}},
\frac{\ve_{ij}}{\ve_{jl}}, \frac{\ve_{ij}}{\ve_{jk}},
\frac{\ve_{ik}}{\ve_{kl}}, \frac{\ve_{il}}{\ve_{jl}}\right)
\end{equation}
is a solution to \eqref{6.4} with $\Gamma =A_G^*$.
Recall that by Lemma \ref{L-finitely-generated},
the group $A_G^*$ is finitely generated.

We prove the following Proposition concerning \eqref{6.4}.

\begin{proposition}\label{P6.1}
Let $G$ be a field of characteristic $0$ and $\Gamma$ a finitely generated
subgroup of $G^*$. Then there is a finite subset $\mathcal{S}$ of $\Gamma$
with $1\in\mathcal{S}$ such that for every solution
$(x_1\kdots y_3)\in\Gamma^6$ of \eqref{6.4},
at least one of the following holds:
\\[0.2cm]
{\bf (i)} at least one of $x_1\kdots y_3$ belongs to $\mathcal{S}$;
\\
{\bf (ii)} there are $\eta_1,\eta_2,\eta_3\in\{ \pm 1\}$ such that
$(y_1,y_2,y_3)$ is a permutation of $(x_1^{\eta_1},x_2^{\eta_2},x_3^{\eta_3})$;
\\
{\bf (iii)} one of the numbers in
$\{x_ix_j,\, x_i/x_j,\, y_iy_j,\, y_i/y_j:\, 1\leq i<j\leq 3\}$
is equal to either $-1$, or to a primitive cube root of unity.
\end{proposition}

We remark here that case (iii) may occur. For instance, let $i^2=-1$, let
$\rho$ denote a primitive cube root of unity, and assume that
$i,\rho\in\Gamma$. Then for every $u\in\Gamma$, the tuple
$(u^6,iu^3,-iu^3,u^4,\rho u^4,\rho^2u^4)$ satisfies \eqref{6.4}.
There are various other such infinite families of solutions.
Proposition \ref{P6.1} contains only the information needed for the proof
of Theorem \ref{T2.2}.

Proposition \ref{P6.1} is deduced from the following lemma.
Here and below, $T$ is an auxiliary variable,
and by $\equiv$ we indicate that an identity holds identically in $T$.

\begin{lemma}\label{L6.2}
Let $G$, $\Gamma$ be as in Proposition \ref{P6.1}.
Let $m,n$ be non-negative integers with $m+n>0$.
Then there is a finite subset $\mathcal{T}$ of $\Gamma$
with $1\in\mathcal{T}$
such that for every solution
$(x_1\kdots x_m,\, y_1\kdots y_n,z)\in\Gamma^{m+n+1}$ of
\begin{equation}\label{6.5}
(1-x_1)\cdots (1-x_m)\, =\, z(1-y_1)\cdots (1-y_n),
\end{equation}
at least one of the following
holds:
\\[0.2cm]
{\bf (i)} at least one of $x_1\kdots y_n$ belongs to $\mathcal{T}$;
\\
{\bf (ii)} there are integers $c_1\kdots c_m,\, d_1\kdots d_n, e$
with $c_1\cdots c_md_1\cdots d_n\not= 0$,
such that
\begin{equation}\label{6.6}
(1-x_1T^{c_1})\cdots (1-x_mT^{c_m})\equiv zT^{e}(1-y_1T^{d_1})\cdots (1-y_nT^{d_n}).
\end{equation}
\end{lemma}

\begin{proof}
We proceed by induction on $m+n$. For $m=1,n=0$, say, our assertion is a simple consequence of the fact
that the equation $1-x_1=z$ has only finitely many solutions in $x_1,z\in\Gamma$. Let $p\geq 2$,
and suppose that the lemma is true for all pairs of non-negative integers $m,n$ with $m+n<p$.
Take non-negative integers $m,n$ with $m+n=p$.
By Lemma \ref{L0}, for all but finitely many solutions $(x_1\kdots y_n,z)\in\Gamma^{m+n+1}$ of \eqref{6.5} with $x_i\not= 1$ for $i=1\kdots m$,
$y_j\not= 1$ for $j=1\kdots n$,
there are integers $c_1\kdots d_n,e$, not all $0$, such that \eqref{6.6} holds,
but some of $c_1\kdots c_m,d_1\kdots d_n$ may be zero. Notice that \eqref{6.6} cannot hold
with $e\not= 0$ and all $c_i,d_j$ equal to $0$.
Fix a solution $(x_1\kdots y_n,z)$ satisfying \eqref{6.6} where some of the $c_i,d_j$ are $0$,
and put $I:=\{ i :\, c_i\not= 0\}$, $I^c :=\{ 1\kdots m\}\setminus I$,
$J:=\{ j :\, d_j\not= 0\}$, $J^c :=\{ 1\kdots n\}\setminus J$.
Then at least one of $I,J$ is non-empty.

For $i\in I$, put $a_i :=|c_i|$
and $u_i := x_i^{\pm 1}$ with $u_i^{a_i}=x_i^{-c_i}$.
Likewise, for $j\in J$, put $b_j:=|d_j|$, and $v_j:=y_j^{\pm 1}$ such that $v_j^{b_j}=y_j^{-d_j}$.
Then \eqref{6.5} implies that
\[
\prod_{i\in I} (T^{a_i}-u_i)\cdot \prod_{i\in I^c} (1-x_i)\equiv
z'T^f\prod_{j\in J} (T^{b_j}-v_j)\cdot \prod_{j\in I^c} (1-y_j)
\]
with $z'\in\Gamma$, $f\in\Zz$. Since both sides of this identity must be polynomials
with equal leading coefficients, we have $f=0$, and
\begin{equation}\label{6.7}
\prod_{i\in I^c} (1-x_i)=z'\prod_{j\in I^c} (1-y_j).
\end{equation}
By combining this with \eqref{6.6} we obtain
\begin{equation}\label{6.8}
\prod_{i\in I} (1-T^{c_i}x_i)\equiv z''T^e \prod_{j\in J}(1-T^{d_j}y_j),
\end{equation}
where $z' z'' =z$. Recall that all but finitely many solutions of \eqref{6.5}
satisfy both \eqref{6.7}, \eqref{6.8}.

We apply the induction hypothesis to \eqref{6.7}. Notice that $|I^c|+|J^c|<m+n$ since at least
one of the sets $I,J$ is non-empty.
It follows that there
exists a finite set $\mathcal{T}'$ with $1\in\mathcal{T}'$ such that for every tuple $(x_i:\, i\in I^c;\, y_j:\, j\in J^c; z ')$
with entries from $\Gamma$, satisfying \eqref{6.7},
either one of the $x_i$ ($i\in I^c$) or $y_j$ ($j\in J^c$) belongs to $\mathcal{T}'$, or there
are integers
$c_i\, (i\in I^c)$, $d_j:\, (j\in J^c)$, $e '$ with $\prod_{i\in I^c} c_i\prod_{j\in J^c} d_j\not= 0$ such that
\[
\prod_{i\in I^c} (1-x_iT^{c_i})
\equiv z'T^{e'}\prod_{j\in I^c} (1-y_jT^{d_j}).
\]
By multiplying this with \eqref{6.8},
we obtain an identity of the type \eqref{6.6}
where none of the $c_i,\, d_j$ are $0$.
All solutions $(x_1\kdots x_m;\, y_1\kdots y_n;\, z)\in\Gamma^{m+n+1}$
of \eqref{6.5}
satisfy this identity,
except those
for which some $x_i$ or $y_j$ belongs to $\mathcal{T}'$
or the finitely many solutions with all $x_i$, $y_j$ different from $1$
for which \eqref{6.7}, \eqref{6.8} do not both hold.
This completes our induction step, and our proof.
\end{proof}

\begin{proof}[Proof of Proposition \ref{P6.1}]
We take for $\mathcal{S}$ the set $\mathcal{T}$ from Lemma \ref{L6.2}, taken with $m=n=3$ and $z=1$.
Pick a solution $(x_1\kdots y_3)\in\Gamma^6$ of \eqref{6.4} with none of the $x_i,y_j$ in $\mathcal{S}$.
Then there are integers $c_1\kdots d_3$ and $e$ with $c_1c_2c_3d_1d_2d_3\not= 0$ such that
\begin{equation}\label{6.9}
(1-x_1T^{c_1})(1-x_2T^{c_2})(1-x_3T^{c_3})\equiv T^e(1-y_1T^{d_1})(1-y_2T^{d_2})(1-y_3T^{d_3}).
\end{equation}
For $i=1,2,3$, define $a_i:=|c_i|$, $b_i:=|d_i|$, $u_i:=x_i^{\pm 1}$, $v_i:=y_i^{\pm 1}$,
where $u_i^{a_i}=x_i^{-c_i}$, $v_i^{b_i}=y_i^{-d_i}$.
Then \eqref{6.9} can be rewritten as
an identity in polynomials
\begin{equation}\label{6.10}
(T^{a_1}-u_1)(T^{a_2}-u_2)(T^{a_3}-u_3)\equiv (T^{b_1}-v_1)(T^{b_2}-v_2)(T^{b_3}-v_3)
\end{equation}
with positive integers $a_1\kdots b_3$;
here we have divided out possible powers of $T$ on both sides.

In what follows we assume that
\begin{equation}\label{6.11}
u_i+u_j\not=0,\ \ v_i+v_j\not=0\ \mbox{for } 1\leq i<j\leq 3
\end{equation}
and prove that at least one of the following two alternatives must hold:
\begin{eqnarray}
\label{6.12}
&(v_1,v_2,v_3)\ \mbox{is a permutation of } (u_1,u_2,u_3);&
\\[0.1cm]
\label{6.13}
&
\begin{array}{l}
\{ u_i/u_j,\, v_i/v_j\ (i\leq i<j\leq 3)\}
\\
\qquad\qquad\qquad\mbox{contains a primitive cube root of unity.}
\end{array}&
\end{eqnarray}
This clearly implies Proposition \ref{P6.1}. Since \eqref{6.11}--\eqref{6.13}
are invariant under permutations of $u_1,u_2,u_3$, under permutations of $v_1,v_2,v_3$
and under interchanging the tuples $(u_1,u_2,u_3)$, $(v_1,v_2,v_3)$, it suffices to consider
the cases (i)--(x) below.
\\[0.1cm]

\noindent
{\bf Case (i).} $a_1>a_2>a_3,\ \ b_1>b_2>b_3$.
\\
Then \eqref{6.10} becomes
\begin{eqnarray*}
&&T^{a_1+a_2+a_3}-u_3T^{a_1+a_2}-u_2T^{a_1+a_3}-u_1T^{a_2+a_3}
\\
&&\qquad\qquad +u_2u_3T^{a_1}+u_1u_3T^{a_2}+u_1u_2T^{a_3}-u_1u_2u_3\ \equiv
\\[0.1cm]
&&T^{b_1+b_2+b_3}-v_3T^{b_1+b_2}-v_2T^{b_1+b_3}-v_1T^{b_2+b_3}
\\
&&\qquad\qquad +v_2v_3T^{b_1}+v_1v_3T^{b_2}+v_1v_2T^{b_3}-v_1v_2v_3.
\end{eqnarray*}
We have either $a_2+a_3\not=a_1$ and $b_2+b_3\not= b_1$ or $a_2+a_3=a_1$ and $b_2+b_3=b_1$.
But in each of these cases,
the second largest exponent on $T$ on the left is $a_1+a_2$ and that on the right $b_1+b_2$; hence $u_3=v_3$.
Likewise, the third largest exponent
on $T$ on the left is $a_1+a_3$ and that on the right $b_1+b_3$; so $u_2=v_2$.
Finally, $u_1u_2u_3=v_1v_2v_3$; hence $u_1=v_1$. This implies
\eqref{6.12}.
\\[0.2cm]
{\bf Case (ii).} $a_1>a_2>a_3,\ \ b_1=b_2>b_3$.
\\
Then \eqref{6.10} becomes
\begin{eqnarray*}
&&T^{a_1+a_2+a_3}-u_3T^{a_1+a_2}-u_2T^{a_1+a_3}-u_1T^{a_2+a_3}+u_2u_3T^{a_1}
\\
&&\qquad\qquad +u_1u_3T^{a_2}+u_1u_2T^{a_3}-u_1u_2u_3\ \equiv
\\[0.1cm]
&&T^{2b_1+b_3}-v_3T^{2b_1}-(v_1+v_2)T^{b_1+b_3}
\\
&&\qquad\qquad +(v_1+v_2)v_3T^{b_1}+v_1v_2T^{b_3}-v_1v_2v_3.
\end{eqnarray*}
By \eqref{6.11}, the right-hand side consists of $6$ terms with different
exponents on $T$ and non-zero coefficients.
So on the left-hand side, two terms have to cancel each other
and this is possible
only if $a_2+a_3=a_1$ and $u_1=u_2u_3$.
Comparing the remaining term with the largest exponent on $T$ on the left
with the term with the largest exponent on $T$ on the right, and
also the terms on both sides with the second largest, third largest
exponent on $T$, etc.,
we see that $a_1+a_2+a_3=2b_1+b_3$, $a_1+a_2=2b_1$,
$a_1+a_3=b_1+b_3$. This implies $a_3=b_3$, $a_1=b_1$, $a_2=b_1$, contradicting $a_1>a_2$.
So Case (ii) is impossible.
\\[0.2cm]
{\bf Case (iii).} $a_1>a_2>a_3,\ \ b_1>b_2=b_3$.
\\
Then \eqref{6.10} becomes
\begin{eqnarray*}
&&T^{a_1+a_2+a_3}-u_3T^{a_1+a_2}-u_2T^{a_1+a_3}-u_1T^{a_2+a_3}+u_2u_3T^{a_1}
\\
&&\qquad\qquad +u_1u_3T^{a_2}+u_1u_2T^{a_3}-u_1u_2u_3\ \equiv
\\[0.1cm]
&&T^{b_1+2b_3}-(v_2+v_3)T^{b_1+b_3}-v_1T^{2b_3}
\\
&&\qquad\qquad +v_2v_3T^{b_1}+v_1(v_2+v_3)T^{b_3}-v_1v_2v_3.
\end{eqnarray*}
Again, on the left-hand side we must have cancellation of two terms, implying $a_2+a_3=a_1$ and $u_1=u_2u_3$.
On the right-hand side, all six terms must have different exponents on $T$, so $2b_3\not= b_1$. If $2b_3>b_1$,
then comparing on both sides the three terms with the largest powers
of $T$, we get $a_1+a_2+a_3=b_1+2b_3$, $a_1+a_2=b_1+b_3$, $a_1+a_3=2b_3$, implying $a_1=a_3=b_3$ which is impossible.
So $b_1>2b_3$. Then comparing the exponents on $T$
of the corresponding terms on the left-
and right-hand side does not lead to a contradiction.
Comparing the coefficients of the terms with the second largest exponent on $T$, i.e.,
with $T^{a_1+a_2}$, $T^{b_1+b_3}$, with the
third largest exponent, etc.,
we get $u_3=v_2+v_3$, $u_2=-v_2v_3$, $u_1u_3=-v_1$, $u_1u_2=v_1(v_2+v_3)$,
 $u_1u_2u_3=v_1v_2v_3$.
Consequently, $v_1v_2v_3=u_1u_2u_3=v_1(v_2+v_3)^2$, hence $v_2v_3=(v_2+v_3)^2$, $v_2^2+v_2v_3+v_3^2=0$,
$v_2/v_3$ is a primitive cube root of unity. This implies \eqref{6.13}.
\\[0.2cm]
{\bf Case (iv).} $a_1>a_2>a_3,\ \ b_1=b_2=b_3$.
\\
In this case, the expansion of the left-hand side of \eqref{6.10} gives at least $6$ non-zero terms
with distinct powers of $T$,
while the right-hand side cannot have more than $4$ terms.
So this case is impossible.
\\[0.2cm]
{\bf Case (v).} $a_1=a_2>a_3,\ \ b_1=b_2>b_3$.
\\
Then \eqref{6.10} becomes
\begin{eqnarray*}
&&T^{2a_1+a_3}-u_3T^{2a_1}-(u_1+u_2)T^{a_1+a_3}
\\
&&\qquad\qquad +(u_1+u_2)v_3T^{a_1}+u_1u_2T^{a_3}-a_1a_2a_3\ \equiv
\\[0.1cm]
&&T^{2b_1+b_3}-v_3T^{2b_1}-(v_1+v_2)T^{b_1+b_3}
\\
&&\qquad\qquad +(v_1+v_2)v_3T^{b_1}+v_1v_2T^{b_3}-v_1v_2v_3.
\end{eqnarray*}
By \eqref{6.11} we have on both sides $6$ non-zero terms with distinct
powers of $T$.
Comparing the terms on both sides with the second highest power of $T$, i.e.,
$T^{2a_1}$ and $T^{2b_1}$, we get $u_3=v_3$.
Comparing the terms with the third highest power of $T$, i.e.,
$T^{a_1+a_3}$ and $T^{b_1+b_3}$, we obtain $u_1+u_2=v_1+v_2$,
and finally, from the terms with the smallest positive power of $T$,
i.e., $T^{a_3}$, $T^{b_3}$, we obtain $u_1u_2=v_1v_2$.
Hence $\{ u_1,u_2\}=\{ v_1,v_2\}$. This implies \eqref{6.12}.
\\[0.2cm]
{\bf Case (vi).} $a_1=a_2>a_3$,\ \ $b_1>b_2=b_3$.
\\
Then \eqref{6.10} becomes
\begin{eqnarray*}
&&T^{2a_1+a_3}-u_3T^{2a_1}-(u_1+u_2)T^{a_1+a_3}
\\
&&\qquad\qquad +(u_1+u_2)u_3T^{a_1}+u_1u_2T^{a_3}-u_1u_2u_3\ \equiv
\\[0.1cm]
&&T^{b_1+2b_3}-(v_2+v_3)T^{b_1+b_3}-v_1T^{2b_3}
\\
&&\qquad\qquad +v_2v_3T^{b_1}+v_1(v_2+v_3)T^{b_3}-v_1v_2v_3.
\end{eqnarray*}
On the left-hand side there are $6$ non-zero terms with distinct powers of $T$.
So on the right-hand side we must also have $6$ non-zero terms
with distinct powers of $T$. We have either $2b_3>b_1$ or $2b_3<b_1$.
If $2b_3>b_1$ then, on comparing the terms with the three largest
exponents on $T$ on both sides we get
$2a_1+a_3=b_1+2b_3$, $2a_1=b_1+b_3$, $a_1+a_3=2b_3$,
hence $a_1=a_3=b_3$, which is impossible. So $b_1>2b_3$.
Then comparing the coefficients of the terms with the largest exponent on $T$
on both sides, the terms with the second largest exponent, etc.
we get $u_3=v_2+v_3$, $u_1+u_2=v_3$, $(u_1+u_2)u_3=v_2v_3$,
$u_1u_2=v_1(v_2+v_3)$, $u_1u_2u_3=v_1v_2v_3$. This leads to
$v_1v_2v_3=v_1(v_2+v_3)^2$, and then similarly as in Case (iii)
it follows that
$v_2/v_3$ is a primitive cube root of unity. Hence \eqref{6.13} holds.
\\[0.2cm]
{\bf Case (vii).} $a_1=a_2>a_3,\ \ b_1=b_2=b_3$.
\\
This case is impossible since on the left-hand side of \eqref{6.10}
we have $6$ non-zero terms with distinct powers of $T$
and on the right-hand side not more than $4$ terms.
\\[0.2cm]
{\bf Case (viii).} $a_1>a_2=a_3,\ \ b_1>b_2=b_3$.
\\
Then \eqref{6.10} becomes
\begin{eqnarray*}
&&T^{a_1+2a_3}-(u_2+u_3)T^{a_1+a_3}-u_1T^{2a_3}
\\
&&\qquad\qquad +u_2u_3T^{a_1}+u_1(u_2+u_3)T^{a_3}-u_1u_2u_3\ \equiv
\\[0.1cm]
&&T^{b_1+2b_3}-(v_2+v_3)T^{b_1+b_3}-v_1T^{2b_3}
\\
&&\qquad\qquad +v_2v_3T^{b_1}+v_1(v_2+v_3)T^{b_3}-v_1v_2v_3.
\end{eqnarray*}
There are various possibilities depending on
whether $2a_3=a_1$, $2a_3\not= a_1$,
$u_1=u_2u_3$, $u_1\not= u_2u_3$ and similarly for the $b_i$'s and $v_i$'s.
But in each of these cases, $a_1+a_3$ is the second largest exponent on $T$
occurring on the
left and $b_1+b_3$ the second largest exponent on the right and so
$u_2+u_3=v_2+v_3$. Further, $a_3$ is the smallest positive exponent on the
left and $b_3$ the smallest positive exponent on the right and so
$u_1(u_2+u_3)=v_1(v_2+v_3)$; and finally $u_1u_2u_3=v_1v_2v_3$.
It follows that $u_1=v_1$, $u_2u_3=v_2v_3$, and then $\{ u_2,u_3\}=\{ v_2,v_3\}$.
This implies \eqref{6.12}.
\\[0.2cm]
{\bf Case (ix).} $a_1>a_2=a_3,\ \ b_1=b_2=b_3$.
\\
Then \eqref{6.10} becomes
\begin{eqnarray*}
&&T^{a_1+2a_3}-(u_2+u_3)T^{a_1+a_3}-u_1T^{2a_3}+u_2u_3T^{a_1}
\\
&&\qquad\qquad +u_1(u_2+u_3)T^{a_3}-u_1u_2u_3\ \equiv
\\[0.1cm]
&&T^{3b_1}-(v_1+v_2+v_3)T^{2b_1}+(v_2v_3+v_1v_3+v_1v_2)T^{b_1}-v_1v_2v_3.
\end{eqnarray*}
Then necessarily, $a_1=2a_3$ and $u_1=u_2u_3$. Further,
all terms on the right-hand side are non-zero.
Comparing the terms with the largest and second largest exponent on $T$,
we see that $a_1+2a_3=3b_1$, $a_1+a_3=2b_1$, hence $a_1=a_3=b_1$
which is impossible.
\\[0.2cm]
{\bf Case (x).} $a_1=a_2=a_3,\ \ b_1=b_2=b_3$.
\\
Then \eqref{6.10} implies at once \eqref{6.12}. This completes the proof
of Proposition \ref{P6.1}.
\end{proof}

\section{Proof of Theorem \ref{T2.2}}\label{7}
\setcounter{equation}{0}

Let as before $A$ be an integrally closed domain with quotient field $L$ of characteristic $0$
which is finitely generated
over $\Zz$, and $K$ an extension of $L$ of finite degree $d\geq 3$.
Further, denote by $G$ the normal closure of $K$ over $L$.
In what follows, we consider pairs $(\alpha ,\beta )$ such that
\begin{equation}\label{7.0}
\left\{
\begin{array}{ll}
\mbox{$L(\alpha )=L(\beta )=K$,\ \ $\alpha ,\beta$ are integral over $A$,}
\\
\mbox{$A[\alpha ]= A[\beta ]$,\ \ $\alpha ,\beta$ are not $A$-equivalent.}
\end{array}
\right.
\end{equation}

The next lemma implies part (i) of Theorems \ref{T2.2} and \ref{T1.2}.

\begin{lemma}\label{L7.2}
Suppose that $[K:L]=3$. Let $(\alpha ,\beta )$ be a pair with
\eqref{7.0}.
Then there is a matrix
$\left(\begin{smallmatrix}a_1&a_2\\ a_3&a_4\end{smallmatrix}\right)$
such that
\begin{equation}\label{7.9}
\Big(\begin{array}{cc}a_1&a_2\\ a_3&a_4\end{array}\Big)\in\GL (2,L),
\ \ \
\beta =\frac{a_1\alpha +a_2}{a_3\alpha +a_4},\ \ a_3\not= 0.
\end{equation}
Further, if $A$ is a principal ideal domain then
$\left(\begin{smallmatrix}a_1&a_2\\ a_3&a_4\end{smallmatrix}\right)$ can be chosen
from $\GL (2,A)$.
\end{lemma}

\begin{proof}
Let $\sigma_i$ ($i=1,2,3$) be the $L$-isomorphisms of $K$ into $G$,
and write $\alpha^{(i)}:=\sigma_i(\alpha )$, $\beta^{(i)}:=\sigma_i(\beta )$
for $i=1,2,3$.
By straightforward linear algebra, there are $a_1,a_2,a_3,a_4\in G$ such that
\[
\beta^{(i)}=\frac{a_1\alpha^{(i)}+a_2}{a_3\alpha^{(i)}+a_4}\ \
\mbox{for $i=1,2,3$.}
\]
If we choose the first non-zero element among $a_1\kdots a_4$ equal to $1$,
then $a_1\kdots a_4$ are uniquely determined.
By applying $\sigma\in {\rm Gal}(G/L)$ and observing that $\sigma$ permutes
the $\alpha^{(i)}$ in the same way as the $\beta^{(i)}$, we infer that
$\sigma (a_i)=a_i$ for $i=1\kdots 4$. Hence $a_i\in L$ for $i=1\kdots 4$.
The matrix $\left(\begin{smallmatrix}a_1&a_2\\ a_3&a_4\end{smallmatrix}\right)$
must have non-zero determinant since otherwise
$\beta^{(1)}=\beta^{(2)}=\beta^{(3)}$, contrary to our assumption
$L(\beta )=K$. Next, we must have $a_3\not= 0$.
For otherwise, $\alpha ,\beta$ are $L$-equivalent,
hence $A$-equivalent by Lemma \ref{L2b},
contrary to our assumptions.
This proves \eqref{7.9}.

In case that $A$ is a principal ideal domain, by taking a scalar multiple
of $\left(\begin{smallmatrix}a_1&a_2\\ a_3&a_4\end{smallmatrix}\right)$,
we can see to it
that $a_1\kdots a_4\in A$ and $(a_1\kdots a_4)=(1)$.
Then
$\left(\begin{smallmatrix}a_1&a_2\\ a_3&a_4\end{smallmatrix}\right)\in\GL (2,A)$
by Lemma \ref{L7.1}. This completes the proof of Lemma \ref{L7.2}.
\end{proof}

In what follows, we assume that
\begin{equation}\label{7.10}
\left\{
\begin{array}{l}
\mbox{$[K:L]=d\geq 4$,\ \ ${\rm Gal}(G/L)\cong S_4$ if $d=4$,}
\\
\mbox{$K$ is four times transitive over $L$ if $d\geq 5$.}
\end{array}
\right.
\end{equation}
For every pair $(\alpha ,\beta )$ with \eqref{7.0} we define, in the usual
manner,
\[
\ve_{ij}:=\frac{\alpha^{(i)}-\alpha^{(j)}}{\beta^{(i)}-\beta^{(j)}}
\ \ (1\leq i,j\leq d,\, i\not= j).
\]

We start with a simple, but for our proof important observation.

\begin{lemma}\label{L7.4} Let $\alpha ,\beta$ satisfy \eqref{7.0},
and let $(p_1,p_2,p_3,p_4)$, $(q_1,q_2,q_3,q_4)$ be two ordered tuples
of distinct indices from $\{ 1\kdots d\}$.
Then there is $\sigma\in{\rm Gal}(G/L )$
such that
\[
\mbox{$\sigma (\ve_{p_i,p_j})=\ve_{q_i,q_j}$ for each distinct $i,j\in\{ 1,2,3,4\}$.}
\]
\end{lemma}

\begin{proof}
By \eqref{7.10}, there is $\sigma\in{\rm Gal}(K/L)$ such that
$\sigma (\alpha^{(p_i)})=\alpha^{(q_i)}$ for $i=1,2,3,4$.
The same holds with $\beta$ instead of $\alpha$. This implies the lemma at once.
\end{proof}

Our next observation is that for any pair $(\alpha ,\beta )$
with \eqref{7.0},
\begin{equation}\label{7.14}
\frac{\ve_{ij}}{\ve_{ik}}\not= 1\ \ \mbox{for } i,j,k\in\{ 1\kdots d\},\
\mbox{with $i,j,k$ distinct.}
\end{equation}
Indeed, suppose there are distinct indices $i,j,k$ with $\ve_{ij}=\ve_{ik}$.
Then by Lemma \ref{L7.4} we have $\ve_{1j}=\ve_{12}$ for $j=3\kdots d$.
This implies that $\tau (\alpha )=\tau (\beta )$,
where $\tau (\cdot )$ is given by \eqref{6.5.taudef}.
Now Lemma \ref{L3} (ii) implies that $\alpha ,\beta$ are $A$-equivalent,
contrary to \eqref{7.0}.

\begin{lemma}\label{L7.3}
There is a finite set $\mathcal{E}$
such that for every pair $(\alpha ,\beta )$
with \eqref{7.0}, at least one of the following alternatives holds:
\\[0.2cm]
{\bf (i)} $\ve_{ij}/\ve_{ik}\in\mathcal{E}$ for each ordered triple
$(i,j,k)$ of distinct indices from $\{ 1\kdots d\}$;
\\
{\bf (ii)} $\ve_{ij}\ve_{kl}=\ve_{ik}\ve_{jl}$ for each ordered quadruple
$(i,j,k,l)$ of distinct indices from $\{ 1\kdots d\}$;
\\
{\bf (iii)} $d=4$, and $\ve_{ij}=-\ve_{kl}$ for each permutation
$(i,j,k,l)$ of $(1,2,3,4)$.
\end{lemma}

\begin{proof}
Pick a pair $(\alpha ,\beta )$ with \eqref{7.0}.
We apply Proposition \ref{P6.1} to \eqref{6.4}, with $\Gamma =A_G^*$
and with for $(x_1,x_2,x_3,y_1,y_2,y_3)$ the tuple \eqref{6.3}
with $(i,j,k,l)=(1,2,3,4)$, i.e.,
\begin{equation}\label{7.11}
\left(\frac{\ve_{13}}{\ve_{23}}, \frac{\ve_{14}}{\ve_{34}},
\frac{\ve_{12}}{\ve_{24}}, \frac{\ve_{12}}{\ve_{23}},
\frac{\ve_{13}}{\ve_{34}}, \frac{\ve_{14}}{\ve_{24}}\right).
\end{equation}

Let $\mathcal{S}$ be the finite set from Proposition \ref{6.1}.
Let $\mathcal{E}$ consist of all conjugates over $L$
of the elements from $\mathcal{S}$,
as well as all roots of unity of order up to $18$.

First suppose that alternative (i) of Proposition \ref{P6.1} holds.
Then there are distinct $p,q,r\in\{ 1\kdots 4\}$,
such that $\ve_{pq}/\ve_{pr}\in\mathcal{S}$.
By Lemma \ref{L7.4} we then have
$\ve_{ij}/\ve_{ik}\in\mathcal{E}$ for each triple
$(i,j,k)$ of distinct indices from $\{ 1\kdots d\}$. This
is alternative (i) of our Lemma.

Next, suppose that alternative (ii) of Proposition \ref{P6.1} holds.
Then
\[
\frac{\ve_{13}}{\ve_{23}}\in
\left\{
\frac{\ve_{12}}{\ve_{23}}, \frac{\ve_{13}}{\ve_{34}},
\frac{\ve_{14}}{\ve_{24}}, \frac{\ve_{23}}{\ve_{12}},
\frac{\ve_{34}}{\ve_{13}}, \frac{\ve_{24}}{\ve_{14}}\right\}.
\]
By \eqref{7.14}, $\ve_{13}/\ve_{23}$ cannot be equal to
$\ve_{12}/\ve_{23}$ or $\ve_{13}/\ve_{34}$.
If
$\ve_{13}/\ve_{23}=\ve_{14}/\ve_{24}$, then $\ve_{13}\ve_{24}=\ve_{14}\ve_{23}$.
Then by Lemma \ref{L7.4}
$\ve_{ij}\ve_{kl}=\ve_{ik}\ve_{jl}$ for any four distinct
indices $i,j,k,l\in\{ 1\kdots d\}$.
This is alternative (ii) of our Lemma.

Assume that $\ve_{13}/\ve_{23}=\ve_{23}/\ve_{12}$; then
$\ve_{23}^2=\ve_{12}\ve_{13}$. By Lemma \ref{L7.4},
we have also $\ve_{13}^2=\ve_{12}\ve_{23}$. Hence $(\ve_{23}/\ve_{13})^3=1$.
Again by Lemma \ref{L7.4}, and the fact that $\mathcal{E}$
contains all cube roots of unity, this implies alternative (i) of our Lemma.

Next, assume that $\ve_{13}/\ve_{23}=\ve_{34}/\ve_{13}$. Then
$\ve_{13}^2=\ve_{23}\ve_{34}$. Then by Lemma \ref{L7.4},
$\ve_{23}^2=\ve_{13}\ve_{34}$. This implies again $(\ve_{13}/\ve_{23})^3=1$
and then alternative (i) of our Lemma.

Finally, assume that $\ve_{13}/\ve_{23}=\ve_{24}/\ve_{14}$.
Then $\ve_{13}\ve_{14}=\ve_{23}\ve_{24}$. By Lemma \ref{L7.4},
the same holds after interchanging the indices $2$ and $3$, and also after
interchanging $2$ and $4$; that is,
we have also $\ve_{12}\ve_{14}=\ve_{23}\ve_{34}$ and
$\ve_{13}\ve_{12}=\ve_{34}\ve_{24}$. Multiplying together
the last two identities and dividing by the first, we obtain
$\ve_{12}^2=\ve_{34}^2$,  or $\ve_{12}=\pm\ve_{34}$.
First suppose that $\ve_{12}=\ve_{34}$. Then by Lemma \ref{L7.4},
we have also $\ve_{13}=\ve_{24}$, $\ve_{14}=\ve_{23}$. Substituting this
into \eqref{6.2} with $(i,j,k,l)=(1,2,3,4)$,
we obtain
\[
\left(\frac{\ve_{13}}{\ve_{14}}-1\right)
\left(\frac{\ve_{14}}{\ve_{12}}-1\right)
\left(\frac{\ve_{12}}{\ve_{13}}-1\right)
\, =\,
\left(\frac{\ve_{12}}{\ve_{14}}-1\right)
\left(\frac{\ve_{13}}{\ve_{12}}-1\right)
\left(\frac{\ve_{14}}{\ve_{13}}-1\right).
\]
But this is impossible, since by \eqref{7.14},
both sides are non-zero, and since the left-hand side is the opposite
of the right-hand side.
Hence $\ve_{12}= -\ve_{34}$ and then by Lemma \ref{L7.4}, also
$\ve_{13}=-\ve_{24}$, $\ve_{14}=-\ve_{23}$. If $d\geq 5$, then again
by Lemma \ref{L7.4}, $\ve_{12}=-\ve_{35}$, implying $\ve_{34}=\ve_{35}$,
which is impossible by \eqref{7.14}. Hence $d=4$.
We conclude that alternative (iii) of our Lemma holds.

Finally, suppose that (iii) of Proposition \ref{P6.1} holds.
Then if $(x_1\kdots y_3)$ is the tuple \eqref{7.11} we have that
at least one of the numbers
$x_ix_j , x_i/x_j ,y_iy_j ,\\y_i/y_j$ ($1\leq i<j\leq 3$) is $-1$
or a primitive cube root of unity.
All these possibilities can be combined
by saying that there is a permutation $(i,j,k,l)$ of $(1,2,3,4)$
such that $\ve_{ik}\ve_{il}/\ve_{jk}\ve_{kl}$
or $\ve_{ik}\ve_{kl}/\ve_{il}\ve_{jk}$ is $-1$ or a primitive cube root of unity.
By Lemma \ref{L7.4}, we may replace the indices $i,j,k,l$ by $1\kdots 4$,
respectively. Then
$(\ve_{13}\ve_{14}/\ve_{23}\ve_{34})^6=1$
or $(\ve_{13}\ve_{34}/\ve_{14}\ve_{23})^6=1$.

First suppose that $(\ve_{13}\ve_{14}/\ve_{23}\ve_{34})^6=1$.
Applying again Lemma \ref{L7.4}, the same holds if we interchange
the indices $2$ and $4$, i.e.,
$(\ve_{13}\ve_{12}/\ve_{34}\ve_{23})^6=1$.
As a consequence, $(\ve_{12}/\ve_{14})^6=1$. But then another application
of Lemma \ref{L7.4} implies that $\ve_{ij}/\ve_{ik}\in\mathcal{E}$
for any three distinct indices $i,j,k$, i.e., alternative (i) of our Lemma.

Finally, suppose that
$(\ve_{13}\ve_{34}/\ve_{14}\ve_{23})^6=1$.
By Lemma \ref{L7.4}, interchanging the indices $1$ and $3$,
we get also $(\ve_{13}\ve_{14}/\ve_{34}\ve_{12})^6=1$.
Multiplying the two identities gives $(\ve_{13}^2/\ve_{12}\ve_{23})^6=1$.
Again by Lemma \ref{L7.4}, interchanging the indices $2$ and $3$,
we get $(\ve_{12}^2/\ve_{13}\ve_{23})^6=1$. Then on dividing the last
two identities, we get $(\ve_{13}/\ve_{12})^{18}=1$. A final application
of Lemma \ref{L7.4} leads to $\ve_{ij}/\ve_{ik}\in\mathcal{E}$
for any three distinct indices $i,j,k$, which is alternative (i) of our Lemma.
This completes our proof.
\end{proof}

\begin{proof}[Proof of Theorem \ref{T2.2}, (ii), (iii)]
Consider the two times monogenic $A$-orders
$\OO =A[\alpha ]=A[\beta ]$ in $K$, where $\alpha ,\beta$ satisfy
\eqref{7.0}.

First consider those $A$-orders $\OO$ such that the pair
$(\alpha ,\beta )$ satisfies alternative (i) of Lemma \ref{L7.3}.
Then by \eqref{6.1}, \eqref{7.14},
there is a finite set $\mathcal{F}$ independent
of $\alpha ,\beta$ such that
$\frac{\beta^{(i)}-\beta^{(j)}}{\beta^{(i)}-\beta^{(k)}}\in\mathcal{F}$
for any three distinct $i,j,k\in\{ 1\kdots d\}$.
Hence for the tuple $\tau (\beta )$ defined by \eqref{6.5.taudef}
there are only finitely many possibilities. Then Lemma \ref{L3}
implies that for the $A$-orders $\OO$ under consideration,
the corresponding $\beta$ lie in only finitely many $L$-equivalence classes.
Subsequently, by Lemma \ref{L5}
these $\beta$ lie in only finitely many $A$-equivalence
classes, and
thus there are only finitely many possibilities for the $A$-order $\OO$.

Next, we consider those $A$-orders $\OO =A[\alpha ]=A[\beta ]$ such that
$(\alpha ,\beta )$ satisfies alternative (ii) of Lemma \ref{L7.3}.
Take such a pair $(\alpha ,\beta )$.
By assumption, $\ve_{ij}\ve_{kl}=\ve_{ik}\ve_{jl}$, hence,
in view of \eqref{6.6.1},
\[
\frac{(\beta^{(i)}-\beta^{(j)})(\beta^{(k)}-\beta^{(l)})}
{(\beta^{(i)}-\beta^{(k)})(\beta^{(j)}-\beta^{(l)})}\, =\,
\frac{(\alpha^{(i)}-\alpha^{(j)})(\alpha^{(k)}-\alpha^{(l)})}
{(\alpha^{(i)}-\alpha^{(k)})(\alpha^{(j)}-\alpha^{(l)})}
\]
for every quadruple $(i,j,k,l)$
of distinct indices from $\{ 1\kdots d\}$. In other words,
the cross ratio of any four numbers among the $\alpha^{(i)}$'s
is equal to the cross ratio of the corresponding numbers among the
$\beta^{(i)}$'s. Then by elementary projective geometry,
there is a matrix
$C=\left(\begin{smallmatrix}a_1&a_2\\a_3&a_4\end{smallmatrix}\right)
\in\GL (2,G)$
such that
\[
\beta^{(i)}=\frac{a_1\alpha^{(i)}+a_2}{a_3\alpha^{(i)}+a_4}\ \
\mbox{for } i=1\kdots d.
\]
If we assume that the first non-zero entry among $a_1\kdots a_4$ is $1$,
the matrix $C$ is uniquely determined. Any $\sigma\in{\rm Gal}(G/L)$
permutes the sequences $\alpha^{(1)}\kdots\alpha^{(d)}$ and
$\beta^{(1)}\kdots\beta^{(d)}$ in the same manner,
hence the above relation holds with $\sigma (C)$ instead of $C$; so
$\sigma (C) =C$. It follows that $C\in\GL (2,L)$. We observe that
$a_3\not= 0$. For otherwise, $\alpha ,\beta$ are $L$-equivalent
and then $A$-equivalent by Lemma \ref{L2b},
contrary to \eqref{7.0}.
This shows that
$\OO =A[\alpha ]=A[\beta ]$ is of type I.
Notice that if $A$ is a principal ideal domain, then by taking a suitable scalar multiple of $C$
we can arrange that
$a_1\kdots a_4\in A$ and $(a_1\kdots a_4)=(1)$, and thus,
$C\in\GL (2,A)$ by Lemma \ref{L7.1}.

Finally, we consider those $A$-orders $\OO =A[\alpha ]=A[\beta ]$ such that
$(\alpha ,\beta )$ satisfies alternative (iii) of Lemma \ref{L7.3}; then $d=4$.
Take such a pair $(\alpha ,\beta )$.
By assumption, $\ve_{ij}=-\ve_{kl}$ for every permutation $(i,j,k,l)$ of $(1,2,3,4)$.
Define
\begin{eqnarray*}
&&u_0:=\ve_{12}\ve_{13}\ve_{14},
\\
&&\alpha_0:=\half u_0 (\ve_{12}^{-1}+\ve_{13}^{-1}+\ve_{14}^{-1}),\ \
\beta_0:=\half (\ve_{12}+\ve_{13}+\ve_{14}).
\end{eqnarray*}
By \eqref{7.10}, the group
${\rm Gal}(G/L)$ acts on $\{ \alpha^{(1)}\kdots\alpha^{(4)}\}$ as the full
permutation group. Say that $\sigma (\alpha^{(i)})=\alpha^{(\sigma (i))}$
for $\sigma\in{\rm Gal}(G/L)$, $i=1,2,3,4$.
Then $\sigma (\beta^{(i)})=\beta^{(\sigma (i))}$ for $i=1,2,3,4$
and thus, $\sigma (\ve_{ij})=\ve_{\sigma (i),\sigma (j)}$ for $1\leq i,j\leq 4$, $i\not= j$.
Further,
${\rm Gal}(G/K)$ consists of those $L$-automorphisms that permute
$\alpha^{(2)},\alpha^{(3)},\alpha^{(4)}$ and leave $\alpha =\alpha^{(1)}$ unchanged.
Hence $u_0,\alpha_0,\beta_0$ are invariant
under ${\rm Gal}(G/K)$ and so belong to $K$. But $u_0$ is in fact invariant under
${\rm Gal}(G/L)$, hence belongs to $L$. Notice that
\begin{equation}\label{7.12}
\beta_0^2 = \alpha_0 +r_0,\ \ \alpha_0^2 =u_0\beta_0 +s_0\ \ \mbox{with } r_0,s_0\in L.
\end{equation}
Indeed, \eqref{7.12} holds with
\[
r_0 := \quarter (\ve_{12}^2+\ve_{13}^2+\ve_{14}^2),\ \
s_0 := \quarter u_0^2(\ve_{12}^{-2}+\ve_{13}^{-2}+\ve_{14}^{-2}),
\]
and these $r_0,s_0$ are invariant under ${\rm Gal}(G/L)$.

A straightforward computation gives
\[
\alpha_0^{(2)}=\half u_0(\ve_{21}^{-1}+\ve_{23}^{-1}+\ve_{24}^{-1})=
\half u_0(\ve_{12}^{-1}-\ve_{13}^{-1}-\ve_{14}^{-1})
\]
and similarly, $\beta_0^{(2)}=\frac{1}{2}(\ve_{12}-\ve_{13}-\ve_{14})$. Hence
\[
\frac{\alpha_0^{(1)}-\alpha_0^{(2)}}{\beta_0^{(1)}-\beta_0^{(2)}}=
\frac{-u_0(\ve_{13}^{-1}+\ve_{14}^{-1})}{-(\ve_{13}+\ve_{14})}=u_0\ve_{13}^{-1}\ve_{14}^{-1}=\ve_{12}.
\]
By taking conjugates over $L$ we get
\begin{equation}\label{7.13}
\frac{\alpha_0^{(i)}-\alpha_0^{(j)}}{\beta_0^{(i)}-\beta_0^{(j)}}=\ve_{ij}\ \
\mbox{for } 1\leq i,j\leq 4,\, i\not= j.
\end{equation}
As a consequence, the four conjugates of $\alpha_0$ over $L$ are distinct, and also the four conjugates
of $\beta_0$ over $L$ are all distinct. Hence $L(\alpha_0)=L(\beta_0)=K$. Notice that in the deduction of
\eqref{6.1}, no properties of $\alpha ,\beta$ were used other than that $L(\alpha )=L(\beta )=K$.
That is, the same reasoning applies if we replace $\alpha ,\beta$ by $\alpha_0 ,\beta_0$.
But then, applying \eqref{6.1} both with $(\alpha ,\beta )$ and with $(\alpha_0,\beta_0)$,
using \eqref{7.13}, \eqref{7.14},
we obtain
\[
\frac{\beta^{(i)}-\beta^{(j)}}{\beta^{(i)}-\beta^{(k)}}=
\frac{\beta_0^{(i)}-\beta_0^{(j)}}{\beta_0^{(i)}-\beta_0^{(k)}}
\ \ (1\leq i,j,k\leq d,\  i,j,k\ \mbox{distinct}).
\]
By multiplying this identity with
$\ve_{ij}/\ve_{ik}$ we obtain
\[
\frac{\alpha^{(i)}-\alpha^{(j)}}{\alpha^{(i)}-\alpha^{(k)}}=
\frac{\alpha_0^{(i)}-\alpha_0^{(j)}}{\alpha_0^{(i)}-\alpha_0^{(k)}}
\ \ (1\leq i,j,k\leq d,\  i,j,k\ \mbox{distinct}).
\]
This shows that $\tau (\beta )=\tau (\beta_0)$, $\tau (\alpha )=\tau (\alpha_0)$,
where $\tau (\cdot )$ is defined by \eqref{6.5.taudef}. By Lemma \ref{L3}, (i), there
are $\lambda ,\lambda '\in L^*$, $\mu ,\mu '\in L$, such that
\[
\alpha =\lambda \alpha_0 +\mu ,\ \ \beta =\lambda '\beta_0 +\mu '.
\]
By combining this with \eqref{7.12}, we obtain
\[
\beta =a_0\alpha^2+a_1\alpha +a_2,\ \ \alpha =b_0\beta^2+b_1\beta+b_2
\]
with $a_0,a_1,a_2,b_0,b_1,b_2\in L$ and $a_0b_0\not= 0$.
But in fact, we have $a_0\kdots b_2\in A$
since by assumption, $A[\alpha ]=A[\beta ]$. This shows that $\OO =A[\alpha ]=A[\beta ]$
is an $A$-order of type II. This completes the proof of Theorem \ref{T2.2}.
\end{proof}

\end{document}